\documentclass[11pt]{article}

\usepackage{geometry}   
\usepackage{authblk}

\usepackage[utf8]{inputenc} 
\usepackage{hyperref}       
\usepackage{url}            
\usepackage{booktabs}       
\usepackage{amsfonts}       
\usepackage{nicefrac}       
\usepackage{microtype}      

\usepackage{amsmath,amsfonts,amssymb,amsthm,array}

\usepackage{algorithm}
\usepackage{caption}
\usepackage{subcaption}
\usepackage{algpseudocode}
\usepackage{mdframed} 
\usepackage{thmtools}

\newcommand{\mytextstyle}{}

\newcommand{\R}{\mathbb{R}}

\def\<#1,#2>{\langle #1,#2\rangle}

\newcommand{\prox}{\operatorname{prox}}

\usepackage[colorinlistoftodos,bordercolor=orange,backgroundcolor=orange!20,linecolor=orange,textsize=scriptsize]{todonotes}

\newcommand{\eqdef}{\; { := }\;}

\declaretheorem[style=shaded,within=section]{definition}
\declaretheorem[style=shaded,sibling=definition]{theorem}

\declaretheorem[style=shaded,sibling=definition]{assumption}
\declaretheorem[style=shaded,sibling=definition]{corollary}

\declaretheorem[style=shaded,sibling=definition]{lemma}
\declaretheorem[style=shaded,sibling=definition]{example}

\title{\bf Error Compensated Distributed SGD \\ \bf Can Be Accelerated}

\author[1]{Xun Qian\thanks{website: \url{https://qianxunk.github.io}, email: \texttt{xun.qian@kaust.edu.sa}}}
\author[1]{Peter Richt\'arik\thanks{website: \url{https://richtarik.org}, email: \texttt{peter.richtarik@kaust.edu.sa}}}
\author[2]{Tong  Zhang\thanks{website: \url{http://tongzhang-ml.org}, email:  \texttt{tongzhang@ust.hk}}}

\affil[1]{King Abdullah University of Science and Technology, Thuwal, Saudi Arabia}
\affil[2]{Hong Kong University of Science and Technology, Hong Kong}

\begin{document}
	
	\maketitle
	
\begin{abstract}
Gradient compression is a recent and increasingly popular technique for reducing the communication cost in distributed training of large-scale machine learning models. In this work we focus on developing efficient distributed methods that can work for any compressor satisfying a certain contraction property, which includes both unbiased (after appropriate scaling) and biased compressors such as RandK and TopK. Applied naively, gradient compression introduces errors that either slow down convergence or lead to divergence. A popular technique designed to tackle this issue is error compensation/error feedback. Due to the difficulties associated with analyzing biased compressors, it is not known whether gradient compression with error compensation can be combined with Nesterov's acceleration. In this work, we show for the first time that error compensated gradient compression methods can be accelerated. In particular, we propose and study the error compensated loopless Katyusha method, and establish an accelerated linear convergence rate under standard assumptions. We show through numerical experiments that the proposed method converges with substantially fewer communication rounds than previous error compensated algorithms.
\end{abstract}

\section{Introduction}

In this work we consider the composite finite-sum optimization problem

\begin{equation}\label{primal-LSVRG}
\mytextstyle \min \limits_{x\in \mathbb{R}^d} \left[ P(x) \eqdef \frac{1}{n} \sum\limits_{\tau=1}^n  f^{(\tau)}(x) + \psi(x) \right],
\end{equation}
where $f(x)\eqdef\frac{1}{n}\sum_{\tau}  f^{(\tau)}(x)$ is an average of $n$ smooth\footnote{We say that a function $\phi:\R^d\to \R$ is smooth if it is differentiable, and has $L_\phi$ Lipschitz gradient: $\|\nabla \phi(x)-\phi(y)\| \leq L_\phi  \|x-y\|$ for all $x,y\in \R^d$. We say that $L_\phi$ is the {\em smoothness constant} of $\phi$.} convex functions $f^{(\tau)}:\R^d\to \R$ distributed over $n$ nodes (devices, computers), and  $\psi:\R^d \to \R\cup \{+\infty\}$ is a proper closed convex function representing a possibly nonsmooth regularizer. On each node, $f^{(\tau)}(x)$ is an average of $m$ smooth convex functions 
\begin{equation}\label{eq:bui8f0890f}
\mytextstyle f^{(\tau)}(x) = \frac{1}{m} \sum \limits_{i=1}^m f^{(\tau)}_i(x) , 
\end{equation}
representing the average loss over the training data stored on node $\tau$. While we specifically focus on the case when $m=1$, our results are also new in the $m=1$ case, and hence this regime is relevant as well. We assume throughout that problem (\ref{primal-LSVRG}) has at least one optimal solution $x^*$.   We denote the smoothness constants of functions $f$, $f^{(\tau)}$ and $f^{(\tau)}_i$ using symbols $L_f$, ${\bar L}$ and $L$, respectively. These  constants are in general related as follows: \begin{equation}\label{eq:3smoothnessconstants}L_f \leq {\bar L} \leq n L_f, \qquad \bar{L} \leq L \leq m \bar{L}.\end{equation}

When training very large scale supervised machine learning problems, such as those arising in the context of federated learning \cite{FEDLEARN, FL2017-AISTATS, FEDOPT} (see also recent surveys \cite{FL_survey_2019, FL-big}),  distributed algorithms need to be used. In such settings, communication is generally much slower than  (local) computation, which makes communication the key bottleneck in the design of efficient distributed systems. There are several ways to tackle this issue, including reliance on large mini-batches \cite{Goyal17, You17}, asynchronous learning \cite{Agarwal11, Lian15, Recht11}, local updates \cite{COCOA+journal, localSGD-Stich, localSGD-AISTATS2020, Hanzely2020, Blake2020} and communication compression (e.g., quantization and sparsification) \cite{Alistarh17, Bernstein18, Mish19, Seide14, Wen17}.  In this work we focus on the last of these techniques: communication compression.

\subsection{Communication compression}

\paragraph{Contraction and unbiased compressors.}
We say that a randomized map $Q:\R^d\to \R^d$ is a {\em  contraction compressor} if there exists a constant $0< \delta \leq 1$ such that  
\begin{equation}\label{eq:contractor}
\mathbb{E} \left[\|x - Q(x) \|^2\right] \leq (1-\delta)\|x\|^2, \quad \forall x\in \R^d. 
\end{equation}
Further, we say that a randomized map $\tilde{Q}:\R^d\to \R^d$ is  an {\em unbiased compressor} if  there exists a constant $\omega \geq 0$ such that 
\begin{equation} \label{eq:unbiased}
\mathbb{E}[{\tilde Q}(x) ] = x \quad {\rm and} \quad \mathbb{E}\|{\tilde Q}(x)\|^2 \leq (\omega + 1)\|x\|^2, \qquad \forall x\in \R^d.
\end{equation}

It is well known that (see, e.g., \cite{biased2020}) after appropriate scaling, any unbiased compressor satisfying \eqref{eq:unbiased} becomes a contraction compressor. Indeed, for any ${\tilde Q}$ satisfying \eqref{eq:unbiased}, $\frac{1}{\omega+1}{\tilde Q}$ is  a contraction compressor satisfying \eqref{eq:contractor} with $\delta = \frac{1}{\omega+1}$, as shown here:
\begin{align*}
	 \mytextstyle  \mathbb{E}\left[ \left\|\frac{1}{\omega+1}{\tilde Q}(x) - x \right\|^2 \right]  &= \mytextstyle  \frac{1}{(\omega+1)^2} \mathbb{E} \left[\|{\tilde Q}(x)\|^2\right]+ \|x\|^2 - \frac{2}{\omega+1}\mathbb{E} \left[ \langle {\tilde Q}(x), x\rangle \right]\\ 
	&\leq  \mytextstyle  \frac{1}{\omega+1}\|x\|^2 + \|x\|^2 - \frac{2}{\omega+1}\|x\|^2 = \left(1 - \frac{1}{\omega+1} \right) \|x\|^2. 
\end{align*}

Since compressors are typically applied in a scaled fashion, using a scaling stepsize, this means that for all practical purposes, the class of unbiased compressors is included in the class of contraction compressors. For examples of contraction and unbiased compressors, we refer the reader to \cite{biased2020}.

\subsection{Error compensation} While compression reduces the communicated bits in each communication round, it introduces errors, which generally leads to an increase in the number of communication rounds needed to find a solution of any predefined accuracy.  Still, compression has been found useful in practice, as the trade-off often seems to prefer compression to no compression. In order to deal with the  errors introduced by compression, some form of error compensation/error feedback is needed. 

If we assume that the accumulated  error is bounded, and in the case of  unbiased compressors, the convergence rate of error compensated SGD was shown to be the same as that of vanilla SGD \cite{Tang18}. However, if we only assume bounded second moment of the stochastic gradients, in order to guarantee the boundedness of the accumulated quantization error, some decaying factor needs to be involved in general, and  error compensated SGD is proved to have some advantage over QSGD in some perspective for convex quadratic problem \cite{Wu18}. On the other hand, for contraction compressors (for example, the TopK compressor \cite{Alistarh18}),  error compensated SGD actually has the same convergence rate as vanilla SGD \cite{Stich18, Stich19, Tang19}. Since SGD only has a sublinear convergence rate, the current error compensated methods could not get linear convergence rate. If $f$ is non-smooth and $\psi=0$, error compensated SGD was studied in \cite{karimireddy2019error} in the single node case, and the convergence rate is of order $O\left(  \nicefrac{1}{\sqrt{\delta k} } \right)$. 

For variance-reduced methods,  QSVRG  \cite{Alistarh17} handles the smooth case ($\psi\equiv 0$) and  VR-DIANA \cite{Samuel19} handles the composite  case (general $\psi$). However, the compressors of both algorithms need to be unbiased. Error compensation in VR-DIANA  does not need to be  used since this method successfully employs variance reduction (of the variance introduced by the compressor) instead. In this paper, we study error compensation in conjunction with  the acceleration mechanism employed in  loopless Katyusha (L-Katyusha) \cite{LSVRG}, for any contraction compressor. 

\subsection{Contributions}

We now summarize the main contributions of our work.

\paragraph{Acceleration for error compensation.} We develop a new communication efficient algorithm for solving the distributed optimization problem  \eqref{primal-LSVRG}--\eqref{eq:bui8f0890f} which we call {\em Error Compensated Loopless Katyusha} (ECLK); see Algorithm~\ref{alg:ec-lkatyusha}.  ECLK is the {\em first accelerated} error compensated SGD method, and can be seen as an EC variant of the Loopless Katyusha method developed in \cite{LSVRG}.

\paragraph{Iteration complexity.} We obtain the {\em first accelerated linear convergence rate} for  error compensated methods using contraction operators. The iteration complexity of ECLK is
$$
\mytextstyle O\left(  \left( \frac{1}{\delta} + \frac{1}{p} +  \sqrt{\frac{L_f}{\mu}} + \sqrt{\frac{L}{\mu p n}} + \frac{1}{\delta}\sqrt{\frac{(1-\delta){{\color{blue}\bar L}}}{\mu p}} + \sqrt{\frac{(1-\delta)L}{\mu p \delta}} \right) \log\frac{1}{\epsilon}  \right),
$$
where $p \in (0,1]$ is a parameter of the method described later. This is an improvement over the previous best known result for error compensated SGD by Beznosikov et al.~\cite{biased2020}, who obtain {\em nonaccelerated} linear rate. Moreover, they only consider the special case when $\psi\equiv 0$,  and for their linear rate, they need to assume that $\nabla f^{(\tau)}(x^*) = 0$ for all $\tau$, and that full gradients are computed  by all nodes.

If we invoke additional assumptions (Assumption~\ref{as:expcompressor} or Assumption~\ref{as:topkcompressor}) on the  contraction compressor, the iteration complexity is improved to 
$$
\mytextstyle 
O \left(  \left(  \frac{1}{\delta} + \frac{1}{p} + \sqrt{\frac{L_f}{\mu}}  + \sqrt{\frac{L}{\mu p n}} + \frac{1}{\delta} \sqrt{\frac{(1-\delta) {\color{red}L_f}}{\mu p}} + \sqrt{\frac{(1-\delta) L}{\mu p \delta {\color{red}n}}}  \right) \log \frac{1}{\epsilon} \right). 
$$
This is indeed an improvement since ${\color{red} L_f} \leq {\color{blue}{\bar L}}$ (see \eqref{eq:3smoothnessconstants}), and because of the extra scaling factor of $\color{red}n$ in the last term.
If $\delta=1$, i.e., if no compression is used, we recover the iteration complexity of the accelerated method L-Katyusha \cite{qian2019svrg}. 

\paragraph{Communication complexity.} Considering the communication complexity, the optimal choice of $p$ is $O(r(Q))$, where $r(Q)$ is the {\em compression ratio} for the compressor $Q$ defined in (\ref{eq:defrQ}). In particular, when $L_f = {\bar L} =L$, by choosing the optimal $p$, the communication complexity becomes 
$$
\mytextstyle 
O \left(  \Delta_1 \left(  \frac{r(Q)}{\delta}  + \left( r(Q)  + \frac{\sqrt{r(Q)}}{\sqrt{n}}  + \frac{\sqrt{(1-\delta)r(Q)}}{\delta} \right) \sqrt{\frac{L}{\mu}}   \right) \log \frac{1}{\epsilon}  \right), 
$$
where $\Delta_1$ is the communication cost of the uncompressed vector $x\in \R^d$. 

\section{Gradient Compression Methods}\label{sec:compress}

\subsection{TopK and RandK}
We now give two canonical examples of contraction and unbiased compression operators.

\begin{example}[TopK compressor] For a parameter $1\leq K \leq d$,  the TopK compressor is defined as
$$
({\rm TopK}(x))_{\pi(i)} = \left\{ \begin{array}{rl}
(x)_{\pi(i)}  &\mbox{ if $i\leq K$, } \\
0  \quad \quad &\mbox{ otherwise, }
\end{array} \right.
$$
where $\pi$ is a permutation of $\{1, 2, ..., d\}$ such that $(|x|)_{\pi(i)} \geq (|x|)_{\pi(i+1)}$ for $i = 1, ..., d-1$, and if $(|x|)_{\pi(i)} = (|x|)_{\pi(i+1)}$, then $\pi(i) \leq \pi(i+1)$. 
\end{example}

The definition of {\rm TopK} compressor is slightly different with that of \cite{Stich18}. In this way, {\rm TopK} compressor is a deterministic operator (well-defined when there are equal dimensions).

\begin{example}[RandK compressor] For a parameter $1\leq K \leq d$,  the {\rm RandK} compressor is defined as
$$
({\rm RandK}(x))_{i} = \left\{ \begin{array}{rl}
(x)_i  &\mbox{ if $i \in S$, } \\
0  \quad \quad &\mbox{ otherwise, }
\end{array} \right.
$$
where $S$ is chosen uniformly from the set of all $K$ element subsets of $\{1, 2, ..., d\}$.  {\rm RandK} can be used to define an unbiased compressor via scaling. Indeed, it is easy to see that
\[
\mathbb{E}\left( \frac{d}{K} {\rm RandK}(x)\right)  = x
\]
for all $x\in \R^d$.
\end{example}

For the {\rm TopK} and {\rm RandK} compressors, we have the following property. 

\begin{lemma}[Lemma A.1 in \cite{Stich18}]
	For the {\rm TopK} and {\rm RandK} compressors with $1\leq K \leq d$, we have 
	$$ \mytextstyle 
	\mathbb{E} \left[\|{\rm TopK}(x) - x \|^2 \right] \leq \left(  1 - \frac{K}{d}  \right) \|x\|^2,
	$$
	and
	$$
 	\mathbb{E}\left[ \|{\rm RandK}(x) - x \|^2\right] \leq \left(  1 - \frac{K}{d}  \right) \|x\|^2.  
	$$
\end{lemma}

\subsection{Further assumptions}

We will optionally use the following additional assumptions for the contraction compressor. These assumptions are not necessary, but when used, they will lead to better complexity.

\begin{assumption}\label{as:expcompressor}
	$\mathbb{E}[Q(x)] = \delta x$ and all $x\in \R^d$. 
\end{assumption}

It is easy to verify that {\rm RandK} compressor satisfies Assumption \ref{as:expcompressor} with $\delta = \frac{K}{d}$, and $\frac{1}{\omega+1} {\tilde Q}$, where ${\tilde Q}$ is any unbiased compressor, also satisfies Assumption \ref{as:expcompressor} with $\delta = \frac{1}{\omega+1}$. 

\begin{assumption}\label{as:topkcompressor}
	For $x_{\tau} = \frac{\eta}{{\cal L}_1} g^k_{\tau} + e^k_{\tau} \in \R^d$, $\tau=1, ..., n$ and $k\geq 0$ in Algorithm \ref{alg:ec-lkatyusha}, there exist $\delta^\prime>0$ such that  $\mathbb{E}[Q(x_{\tau})] = Q(x_{\tau})$, and 
	$$\mytextstyle
	\left\|\sum \limits_{\tau=1}^n (Q(x_{\tau}) - x_{\tau} )\right\|^2 \leq (1-\delta^\prime) \left\| \sum \limits_{\tau=1}^n x_{\tau} \right\|^2. 
	$$
\end{assumption}

Since {\rm TopK} is deterministic, we have $\mathbb{E}[Q(x)] = Q(x)$ for any $x\in \R^d$. If $Q(x_{\tau})$ is close to $x_{\tau}$, then $\delta^\prime$ could be larger than $\frac{K}{d}$.  Whenever Assumption \ref{as:topkcompressor} is needed, if $\delta > \delta^\prime$, we could decrease $\delta$ such that $\delta = \min \{  \delta, \delta^\prime  \}$. In this way, we have the uniform parameter $\delta$ for the contraction compressor.

\section{Error Compensated L-Katyusha }

\subsection{Description of the method}

In this section we describe our method:  error compensated L-Katyusha (see Algorithm \ref{alg:ec-lkatyusha}). The search direction in L-Katyusha in the distributed setting ($n\geq 1$) at iteration $k$ is 
\begin{equation}\label{eq:sdinLkatyusha}
\mytextstyle \frac{1}{n} \sum \limits_{\tau=1}^n \left(  \nabla f_{i_k^\tau}^{(\tau)}(x^k) - \nabla f_{i_k^\tau}^{(\tau)}(w^k) + \nabla f^{(\tau)}(w^k) \right), 
\end{equation}
where $i_k^\tau$ is sampled uniformly and independently from $[m] \eqdef \{ 1, 2, ..., m  \}$ on the $\tau$-th node for $1\leq \tau \leq n$, $x^k$ is the current iteration, and $w^k$ is the current reference point. Whenever $\psi$ is nonzero in problem (\ref{primal-LSVRG}), $\nabla f(x^*)$ is nonzero in general, and so is $\nabla f^{(\tau)}(x^*)$. Thus, compressing the direction 
$$
 \nabla f_{i_k^\tau}^{(\tau)}(x^k) - \nabla f_{i_k^\tau}^{(\tau)}(w^k) + \nabla f^{(\tau)}(w^k)
$$ 
directly on each node would cause nonzero noise even if  $x^k$ and $w^k$ converged  to the optimal solution $x^*$. On the other hand, since $f_i^{(\tau)}$ is $L$-smooth, 
$
g^k_\tau = \nabla f_{i_k^\tau}^{(\tau)}(x^k) - \nabla f_{i_k^\tau}^{(\tau)}(w^k) 
$
could be small if  $x^k$ and $w^k$ are close enough. Thus, we compress the vector $\frac{\eta}{{\cal L}_1} g^k_\tau + e^k_\tau$ on each node instead. The accumulated error $e^{k+1}_\tau$ is equal to the compression error at iteration $k$ for each node. On each node, a scalar $u^k_\tau$ is also maintained, and only $u^k_1$ will be updated. The summation of $u^k_\tau$ is $u^k$, and we use $u^k$ to control the update frequency of the reference point $w^k$. All nodes maintain the same copies of $x^k$, $w^k$, $y^k$, $z^k$, ${\tilde g}^k$, and $u^k$. Each node sends their compressed vector ${\tilde g}^k_{\tau} = Q(\frac{\eta}{{\cal L}_1} g^k_{\tau} + e^k_{\tau})$ and $u^{k+1}_\tau$ to the other nodes. If $u^k=1$, each node also sends $\nabla f^{(\tau)}(w^k)$ to the other nodes. After the compressed vector ${\tilde g}^k_{\tau}$ is received, we add $\frac{\eta}{{\cal L}_1} \nabla f(w^k)$ to it as the search direction. We also need the following standard proximal operator: 
$$
\mytextstyle
\prox_{\eta \psi} (x) \eqdef \arg\min_y \left\{  \frac{1}{2}\|x-y\|^2 + \eta \psi(y)  \right\}. 
$$
The reference point $w^k$ will be updated if $u^{k+1}=1$. It is easy to see that $w^k$ will be updated with propobility $p$ at each iteration. 

\begin{algorithm}[tb]
	\caption{Error Compensated Loopless Katyusha (ECLK)}
	\label{alg:ec-lkatyusha}
	\begin{algorithmic}[1]
		\State {\bfseries Parameters:} stepsize parameters $\eta = \frac{1}{3\theta_1} >0$, ${\cal L}_1>0$, $\sigma_1 = \frac{\mu_f}{2{\cal L}_1}\geq 0$, $\theta_1, \theta_2 \in (0, 1)$; probability $p \in (0, 1]$
		\State {\bfseries Initialization:}
		$x^0 = y^0= z^0 = w^0 \in \R^d$; $e^0_\tau = 0 \in \R^d$; $u^0=1\in \R$
		\For{ $k = 0, 1, 2, ...$}
		\For{ $\tau = 1, ..., n$} 
		\State Sample $i_k^\tau$ uniformly and independently in $[m]$ on each node 
		\State $g^k_{\tau} = \nabla f_{i_k^\tau}^{(\tau)}(x^k) - \nabla f_{i_k^\tau}^{(\tau)}(w^k)$ 
		\State ${\tilde g}^k_{\tau} = Q(\frac{\eta}{{\cal L}_1} g^k_{\tau} + e^k_{\tau})$
		\State $e^{k+1}_{\tau} = e^k_{\tau} + \frac{\eta}{{\cal L}_1} g^k_{\tau} - {\tilde g}^k_{\tau}$
		\State $u^{k+1}_\tau = 0$ for $\tau = 2, ..., n$ 
		\State $
		u^{k+1}_1 = \left\{ \begin{array}{rl}
		1 & \mbox{ with probability $p$} \\
		0 &\mbox{ with probability $1-p$}
		\end{array} \right.
		$
		\State Send ${\tilde g}^k_{\tau}$ and $u^{k+1}_\tau$ to the other nodes 
		\State Send $\nabla f^{(\tau)}(w^k)$ to the other nodes if $u^k=1$
		\State Receive ${\tilde g}^k_{\tau}$ and $u^{k+1}_\tau$  from the other nodes
		\State Receive $\nabla f^{(\tau)}(w^k)$ from the other nodes if $u^k=1$ 
		\State ${\tilde g}^k = \frac{1}{n} \sum_{\tau=1}^n {\tilde g}^k_{\tau}$ 
		\State $u^{k+1} = \sum_{\tau=1}^n u^{k+1}_{\tau}$ 
		\State $z^{k+1} = \prox_{\frac{\eta}{(1+\eta\sigma_1) {\cal L}_1} \psi} \left(  \frac{1}{1 + \eta \sigma_1} \left(  \eta \sigma_1 x^k + z^k - {\tilde g}^k - \frac{\eta}{{\cal L}_1} \nabla f(w^k)  \right)  \right)$
		\State $y^{k+1} = x^k + \theta_1 (z^{k+1} - z^k)$
		\State $
		w^{k+1} = \left\{ \begin{array}{rl}
		y^k & \mbox{ if $u^{k+1}=1$ } \\
		w^k &\mbox{ otherwise }
		\end{array} \right.
		$
		\State $x^{k+1} = \theta_1 z^{k+1} + \theta_2 w^{k+1} + (1-\theta_1 - \theta_2)y^{k+1}$
		\EndFor
		\EndFor
	\end{algorithmic}
\end{algorithm}

\subsection{Convergence analysis: preliminaries}
We now introduce some perturbed vectors which will be used in the convergence analysis. In Algorithm~\ref{alg:ec-lkatyusha}, let $e^k = \frac{1}{n}\sum_{\tau=1}^n e^k_\tau$, $g^k = \frac{1}{n} \sum_{\tau=1}^n g^k_\tau$, and ${\tilde x}^k = x^k - \frac{1}{1 + \eta \sigma_1}e^k$, ${\tilde z}^k = z^k - \frac{1}{1+\eta \sigma_1}e^k$ for $k\geq 0$. Then 
$
e^{k+1} = \frac{1}{n} \sum_{\tau=1}^n \left(  e^k_\tau + \frac{\eta}{{\cal L}_1} g^k_\tau - {\tilde g}^k_\tau  \right) = e^k + \frac{\eta}{{\cal L}_1} g^k - {\tilde g}^k, 
$
and 
\begin{eqnarray}
 {\tilde z}^{k+1}
 &=& \mytextstyle z^{k+1} - \frac{1}{1 + \eta\sigma_1}e^{k+1} \nonumber \\
&=& \mytextstyle  \frac{1}{1 + \eta \sigma_1} \left(  \eta\sigma_1 x^k + z^k - {\tilde g}^k - \frac{\eta}{{\cal L}_1}\nabla f(w^k)  \right)  - \frac{\eta \partial \psi(z^{k+1}}{(1 + \eta \sigma_1){\cal L}_1}) - \frac{e^{k+1}}{1 + \eta\sigma_1}  \nonumber \\ 
&=& \mytextstyle  \frac{1}{1 + \eta \sigma_1} \left(  \eta\sigma_1 x^k + z^k - e^k - \frac{\eta}{{\cal L}_1}g^k - \frac{\eta}{{\cal L}_1}\nabla f(w^k)  \right) -  \frac{\eta \partial \psi(z^{k+1})}{(1 + \eta \sigma_1){\cal L}_1}  \nonumber \\ 
&=&\mytextstyle  \frac{1}{1 + \eta \sigma_1} \left(  \eta\sigma_1 {\tilde x}^k + {\tilde z}^k  - \frac{\eta}{{\cal L}_1}g^k - \frac{\eta}{{\cal L}_1}\nabla f(w^k)  \right)  -  \frac{\eta \partial \psi(z^{k+1})}{(1 + \eta \sigma_1){\cal L}_1}. \label{eq:tildezk+1}
\end{eqnarray}

The above relation plays an important role in the convergence analysis, and allows us to follow the analysis of original L-Katyusha. We need the following assumption in this section. 

\begin{assumption}\label{as:eclkatyusha}
	$f_i^{(\tau)}$ is $L$-smooth, $f^{(\tau)}$ is ${\bar L}$-smooth, $f$ is $L_f$-smooth and $\mu_f$-strongly convex, and $\psi$ is $\mu_\psi$-strongly convex.  
\end{assumption}

We define some notations which will be used to construct the Lyapunov functions in the convergence analysis. Define $\mu = \mu_f + \mu_\psi$, ${\tilde {\cal Z}}^k = \frac{{\cal L}_1 + \eta \mu/2}{2\eta} \|{\tilde z}^k - x^*\|^2$, ${\cal Y}^k = \frac{1}{\theta_1} (P(y^k) - P^*)$, and ${\cal W}^k = \frac{\theta_2}{pq\theta_1} (P(w^k) - P^*)$. From the update rule of $w^k$ in Algorithm \ref{alg:ec-lkatyusha}, it is easy to see that 
\begin{equation}\label{eq:wk+1}
\mytextstyle \mathbb{E}_k [{\cal W}^{k+1}] = (1-p) {\cal W}^k + \frac{\theta_2}{q} {\cal Y}^k, 
\end{equation}
for $k \geq 0$.  In the next lemma, we describe the evolution of the terms ${\tilde {\cal Z}}^{k}$ and ${\cal Y}^k$. 

\begin{lemma}\label{lm:zyk+1}
	If ${\cal L}_1 \geq L_f$ and $\theta_1 + \theta_2 \leq 1$, then $\mathbb{E}_k\left[  {\tilde {\cal Z}}^{k+1} + {\cal Y}^{k+1}  \right]$ can be upper bounded by
	\begin{eqnarray*}
\mytextstyle  && \frac{{\cal L}_1{\tilde {\cal Z}}^k}{{\cal L}_1 + \eta\mu/2} + (1-\theta_1 - \theta_2) {\cal Y}^k + pq{\cal W}^k 
  + \left(   \frac{{\cal L}_1}{2\eta} + \frac{\mu_f}{2}  \right)  \|e^k\|^2  + \left(   \frac{{\cal L}_1}{2\eta}  + \frac{\mu}{2}  \right) \mathbb{E}_k \|e^{k+1}\|^2 \\ 
&& \qquad -  \frac{1}{\theta_1} \left(  \theta_2 - \frac{2L}{n{\cal L}_1}  \right) (f(w^k) - f(x^k) - \langle \nabla f(x^k), w^k-x^k \rangle ).
	\end{eqnarray*}
\end{lemma}

Because of the compression, we have the additional error terms $\|e^k\|^2$ and $\|e^{k+1}\|^2$ in the evolution of ${\tilde {\cal Z}}^{k}$ and ${\cal Y}^k$ in Lemma~\ref{lm:zyk+1}. However, from the contraction property of the compressor, we can obtain inequalities controlling the evolution of $\frac{1}{n}\sum_{\tau=1}^n \|e^k_{\tau}\|^2$ and $\|e^k\|^2$ in the following two lemmas. 
\begin{lemma}\label{lm:ek+1}
	Th quantity $ \mathbb{E}_k \left[  \frac{1}{n}\sum \limits_{\tau=1}^n \|e^{k+1}_\tau \|^2   \right]$ is upper bounded by the expression
	\begin{eqnarray*}
 \mytextstyle \left(  1 - \frac{\delta}{2}  \right) \frac{1}{n} \sum \limits_{\tau=1}^n \|e^k_\tau\|^2 + \frac{2(1-\delta)\eta^2}{{\cal L}_1^2} \left(  \frac{2{\bar L}}{\delta} + L  \right) (f(w^k) - f(x^k) - \langle \nabla f(x^k), w^k-x^k \rangle ). 
	\end{eqnarray*}
\end{lemma}

\begin{lemma}\label{lm:ek+1-2}
	Under Assumption \ref{as:expcompressor} or  \ref{as:topkcompressor}, the quantity $\mathbb{E}_k [\|e^{k+1}\|^2] $ is upper bounded by
	\begin{eqnarray*}
 \mytextstyle  \left(  1 - \frac{\delta}{2}  \right) \|e^k\|^2 + \frac{2(1-\delta)\delta}{n^2} \sum \limits_{\tau=1}^n \|e^k_{\tau} \|^2 +  \frac{2(1-\delta)\eta^2}{{\cal L}_1^2} \left(  \frac{2L_f}{\delta} + \frac{3L}{n}  \right)  (f(w^k) - f(x^k) - \langle \nabla f(x^k), w^k-x^k \rangle ).
	\end{eqnarray*}
\end{lemma}

\subsection{Convergence analysis: main results}

From the above three lemmas, we can construct suitable Lyapunov functions which enable us to prove linear convergence. First, we construct the Lyapunov function $\Psi^k$ for the general case as follows.  Let ${\cal L}_2 \eqdef  \frac{4L}{n} +  \frac{112(1-\delta) {\bar L}}{9\delta^2} + \frac{56(1-\delta) L}{9\delta}$, and for $k \geq 0$ define 
$$
\mytextstyle 
{\Phi}^k \eqdef  {\tilde {\cal Z}}^{k} + {\cal Y}^{k} + {\cal W}^{k} + \frac{4{\cal L}_1}{\delta \eta} \cdot \frac{1}{n} \sum_{\tau=1}^n \|e^{k}\|^2. 
$$

We are now ready to state our main convergence theorems.
\begin{theorem}\label{th:eclkatyusha-1}
	Assume the compressor $Q$ in Algorithm \ref{alg:ec-lkatyusha} is a contraction compressor and Assumption~\ref{as:eclkatyusha} holds. If ${\cal L}_1 \geq \max \{  L_f, 3\mu\eta  \}$, $\theta_1 + \theta_2 \leq 1$, and $\theta_2 \geq \frac{{\cal L}_2}{2{\cal L}_1}$, then we have 
	$$ \mytextstyle 
	\mathbb{E} \left[\Phi^k\right]\leq \left(1-\min\left(  \frac{\mu}{\mu+6\theta_1 {\cal L}_1},\theta_1 + \theta_2 - \frac{\theta_2}{q}, p(1-q), \frac{\delta}{6} \right)\right)^k \Phi^0,\enspace \forall k\geq 0.
	$$
\end{theorem}

If Assumption \ref{as:expcompressor} or Assumption \ref{as:topkcompressor} holds, we can define the Lyapunov function $\Psi^k$ as follows. 
Let ${\cal L}_3 \eqdef \frac{4L}{n} + \frac{784(1-\delta) L_f}{9\delta^2} + \frac{56(1-\delta)L}{\delta n}$, and for $k\geq 0$ define 
$$
\mytextstyle 
\Psi^k \eqdef  {\tilde {\cal Z}}^{k} + {\cal Y}^{k} + {\cal W}^{k} + \frac{4{\cal L}_1}{\delta \eta} \|e^{k}\|^2 + \frac{28{\cal L}_1(1-\delta)}{\delta \eta n} \cdot \frac{1}{n} \sum_{\tau=1}^n \|e^{k}_\tau\|^2, 
$$

\begin{theorem}\label{th:eclkatyusha-2}
	Assume the compressor $Q$ in Algorithm \ref{alg:ec-lkatyusha} is a contraction compressor and Assumption~\ref{as:eclkatyusha} holds. Assume Assumption \ref{as:expcompressor} or Assumption \ref{as:topkcompressor} holds. If ${\cal L}_1 \geq \max \{  L_f, 3\mu\eta  \}$, $\theta_1 + \theta_2 \leq 1$, and $\theta_2 \geq \frac{{\cal L}_3}{2{\cal L}_1}$, then we have 
	$$
	\mytextstyle 
	\mathbb{E} \left[\Psi^k\right]\leq \left(1-\min\left(  \frac{\mu}{\mu+6\theta_1 {\cal L}_1},\theta_1 + \theta_2 - \frac{\theta_2}{q}, p(1-q), \frac{\delta}{6} \right)\right)^k \Psi^0,\enspace \forall k\geq 0.
	$$
\end{theorem}

In order to cast the above results into a more digestable form, we formulate the following corollary. 

\begin{corollary}\label{co:eclkatyusha}
	Assume the compressor $Q$ in Algorithm \ref{alg:ec-lkatyusha} is a contraction compressor and Assumption~\ref{as:eclkatyusha} holds. Let 
	$
	{\cal L}_1 = \max\left( {\cal L}_4, L_f, 3\mu \eta \right)$, $\theta_2 = \frac{{\cal L}_4}{2\max\{  L_f, {\cal L}_4  \}}$ and
	\begin{align*} 
	\mytextstyle 
	\theta_1=\left\{\begin{array}{ll}
	\min\left( \sqrt{\frac{\mu}{{\cal L}_4 p}}\theta_2, \theta_2  \right)& \mathrm{~if~}L_f \leq \frac{{\cal L}_4}{p}\\  \min\left( \sqrt{\frac{\mu}{L_f}}, \frac{p}{2}  \right) & \mathrm{otherwise}
	\end{array}\right.  .
	\end{align*}
	(i) Let ${\cal L}_4 = {\cal L}_2$. Then with some $q \in [\frac{2}{3}, 1)$, $\mathbb{E}[\Phi^k] \leq \epsilon \Phi^0$ for 
	\begin{equation}\label{eq:iter1}
	\mytextstyle 
	k \geq O\left(  \left( \frac{1}{\delta} + \frac{1}{p} +  \sqrt{\frac{L_f}{\mu}} + \sqrt{\frac{L}{\mu p n}} + \frac{1}{\delta}\sqrt{\frac{(1-\delta){\bar L}}{\mu p}} + \sqrt{\frac{(1-\delta)L}{\mu p \delta}} \right) \log\frac{1}{\epsilon}  \right). 
	\end{equation}
	
	(ii) Let ${\cal L}_4 = {\cal L}_3$. If Assumption \ref{as:expcompressor} or  \ref{as:topkcompressor} holds, then for some $q \in [\frac{2}{3}, 1)$, we have $\mathbb{E}[\Psi^k] \leq \epsilon \Psi^0$ for 
	\begin{equation}\label{eq:iter2}
\mytextstyle 
	k \geq O \left(  \left(  \frac{1}{\delta} + \frac{1}{p} + \sqrt{\frac{L_f}{\mu}}  + \sqrt{\frac{L}{\mu p n}} + \frac{1}{\delta} \sqrt{\frac{(1-\delta) L_f}{\mu p}} + \sqrt{\frac{(1-\delta) L}{\mu p \delta n}}  \right) \log \frac{1}{\epsilon} \right). 
	\end{equation}
\end{corollary}

Noticing that $L_f\leq {\bar L} \leq nL_f$ and ${\bar L} \leq L \leq m{\bar L}$, the iteration complexity in (\ref{eq:iter2}) could be better than that in (\ref{eq:iter1}). On the other hand, if $L_f = {\bar L} = L$, then both iteration complexities in (\ref{eq:iter1}) and (\ref{eq:iter2}) become 
\begin{equation}\label{eq:iter3}
\mytextstyle 
O\left(   \left(  \frac{1}{\delta} + \frac{1}{p} + \sqrt{\frac{L}{\mu}}  + \sqrt{\frac{L}{\mu p n}} + \frac{1}{\delta} \sqrt{\frac{(1-\delta) L}{\mu p}} \right) \log \frac{1}{\epsilon}   \right). 
\end{equation}

\section{Communication Cost}\label{sec:com}

\paragraph{Optimal choice of $p$.} In Algorithm \ref{alg:ec-lkatyusha}, when $w^k$ is updated, the uncompressed vector $\nabla f^{(\tau)}(w^k)$ need to be communicated. We denote $\Delta_1$ as the communication cost of the uncompressed vector $x\in \R^d$. Define the compress ratio $r(Q)$ for the contraction compressor $Q$ as 
\begin{equation}\label{eq:defrQ}
r(Q) \eqdef \sup_{x \in \R^d} \left\{  \mathbb{E} \left[ \frac{\mbox{ communication cost of $Q(x)$ }}{ \Delta_1} \right] \right\}. 
\end{equation}
Denote the total expected communication cost for $k$ iterations as ${\cal T}_k$. The expected communication cost at iteration $k \geq 1$ is bounded by $\Delta_1 r(Q) + 1 + p\Delta_1$, where 1 bit is needed to communicate $u^{k+1}_\tau$, and the expected communication cost at iteration $k =0$ is bounded by $\Delta_1 r(Q) + 1 + \Delta_1$. Hence, 
\begin{eqnarray}
{\cal T}_k &\leq & \mytextstyle  \Delta_1 r(Q) + 1 + \Delta_1 + (\Delta_1 r(Q) + 1 + p\Delta_1)k \nonumber  \\ 
&\leq &\mytextstyle  \Delta_1 r(Q) + 1 + \Delta_1 + (\Delta_1 r(Q) + 1)\left(  1 + \frac{p}{r(Q)}  \right) k. \label{eq:Tk}
\end{eqnarray}

Next, we discuss how to choose $p$ to minimize the total expected communication cost. From Corollary~\ref{co:eclkatyusha} (i) and (\ref{eq:Tk}), we have $\mathbb{E}[\Phi^k] \leq \epsilon \Phi^0$ for 
\begin{eqnarray*}
{\cal T}_k &=& \mytextstyle O \left(  (\Delta_1r(Q) + 1) \left(  1 + \frac{p}{r(Q)}  \right)  \left(  a + \frac{1}{p} + \frac{b}{\sqrt{p}}  \right) \log \frac{1}{\epsilon}   \right) \\ 
&=& \mytextstyle O \left(  (\Delta_1r(Q) + 1) \left(  a + \frac{pa}{r(Q)} + \frac{1}{p} + \frac{1}{r(Q)} + \frac{b}{\sqrt{p}} + \frac{b\sqrt{p}}{r(Q)}   \right) \log \frac{1}{\epsilon}   \right), 
\end{eqnarray*}
where we denote $a = \frac{1}{\delta} + \sqrt{\frac{L_f}{\mu}}$ and $b = \sqrt{\frac{L}{\mu n}} + \frac{1}{\delta}\sqrt{\frac{(1-\delta){\bar L}}{\mu}} + \sqrt{\frac{(1-\delta)L}{\mu \delta}} $. Noticing that $\frac{b}{\sqrt{p}} + \frac{b\sqrt{p}}{r(Q)} \geq \frac{2b}{\sqrt{r(Q)}}$, we have 
$$
\mytextstyle 
O\left(   a + \frac{pa}{r(Q)} + \frac{1}{p} + \frac{1}{r(Q)} + \frac{b}{\sqrt{p}} + \frac{b\sqrt{p}}{r(Q)}   \right) \geq O\left(  a + \frac{1}{r(Q)} + \frac{b}{\sqrt{r(Q)}} \right), 
$$
and the above lower bound holds for $p = O(r(Q))$. Hence, in order to minimize the total expected communication cost, the optimal choice of $p$ is $O(r(Q))$. 

Under Assumption \ref{as:expcompressor} or  \ref{as:topkcompressor}, from Corollary \ref{co:eclkatyusha} (ii), by the same analysis,  in order to minimize the total expected communication cost for $\mathbb{E}[\Psi^k] \leq \epsilon \Psi^0$, the optimal choice of $p$ is also $O(r(Q))$. 

\paragraph{Comparison to the uncompressed L-Katyusha.} For simplicity, we assume $L_f = {\bar L} = L$ and $\Delta_1 r(Q) \geq O(1)$.  From (\ref{eq:iter3}) and (\ref{eq:Tk}), by choosing $p=O(r(Q))$, we have 
\begin{equation}\label{eq:Tk22}
\mytextstyle 
{\cal T}_k = O \left(  \Delta_1 \left(  \frac{r(Q)}{\delta}  + \left( r(Q)  + \frac{\sqrt{r(Q)}}{\sqrt{n}}  + \frac{\sqrt{(1-\delta)r(Q)}}{\delta} \right) \sqrt{\frac{L}{\mu}}   \right) \log \frac{1}{\epsilon}  \right). 
\end{equation}
For uncompressed L-Katyusha, by choosing $p=1$, we have 
\begin{equation}\label{eq:Tk33}
\mytextstyle 
{\cal T}_k = O \left(  \Delta_1 \sqrt{\frac{L}{\mu}} \log \frac{1}{\epsilon}  \right). 
\end{equation}
If $\frac{\sqrt{r(Q)}}{\delta} < 1$, then the communication cost in (\ref{eq:Tk22}) is less than that in (\ref{eq:Tk33}). For TopK compressor, $r(Q) = \frac{K(64 + \lceil \log d \rceil)}{64d}$, and in practice $\delta$ can be much larger than $\frac{K}{d}$, sometimes even in order $O(1)$.

\begin{figure*}[t]
	\vspace{0cm}
	\centering
	\begin{tabular}{cccc}
		\includegraphics[width=0.45\linewidth]{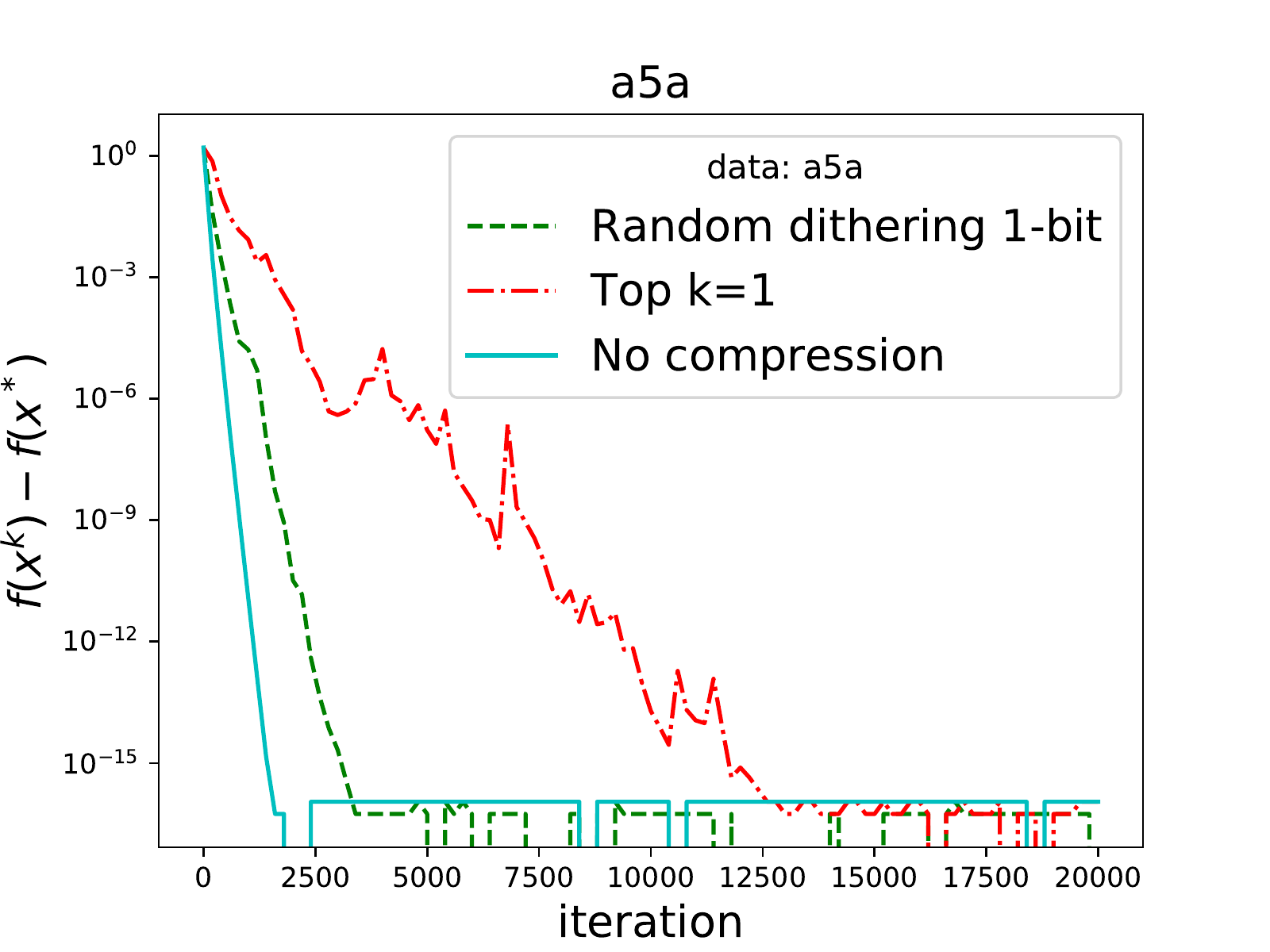}&
		\includegraphics[width=0.45\linewidth]{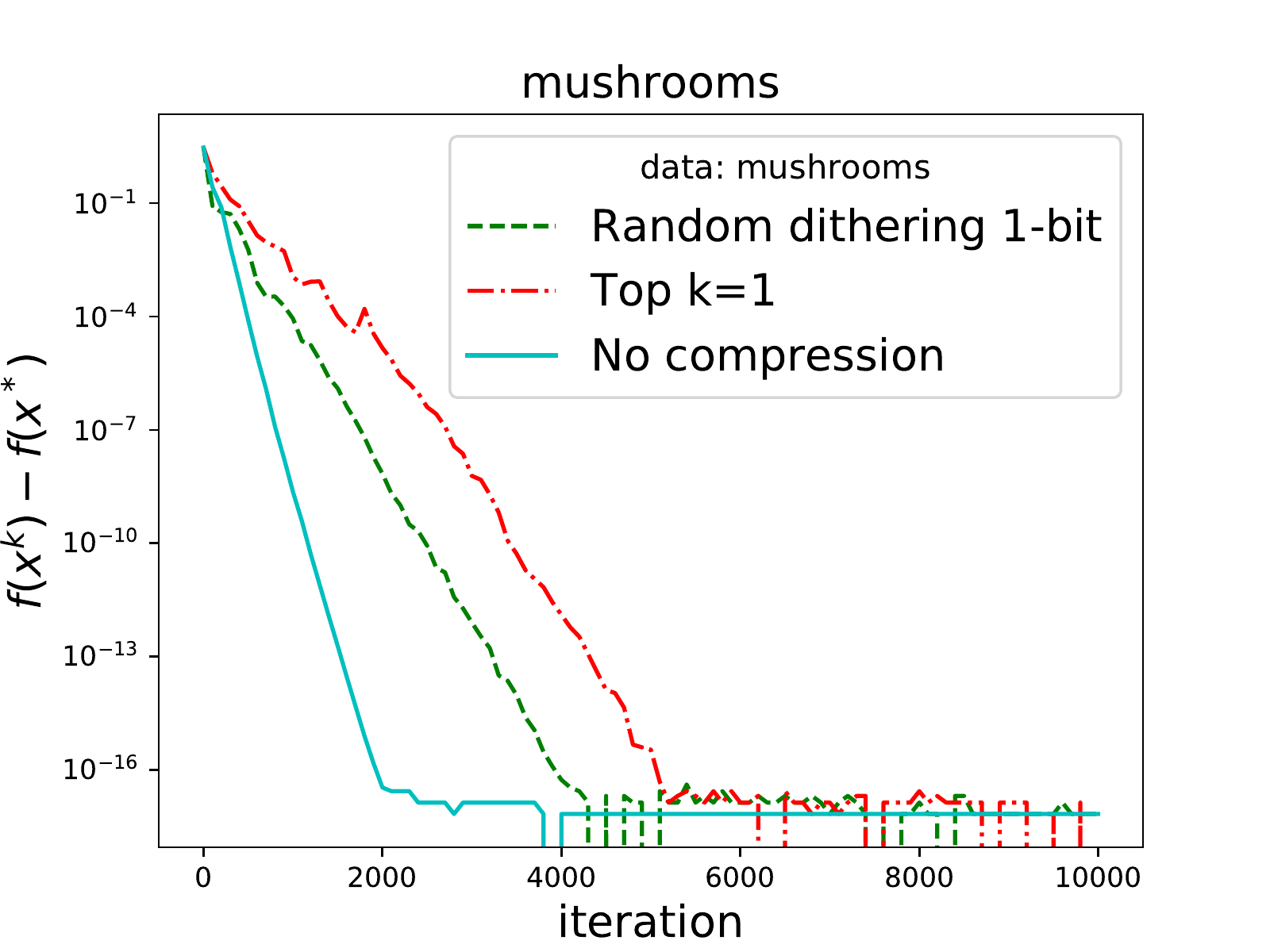}&
		
	\end{tabular}
	\vskip -0.2cm
	\caption{The iteration complexity performance of Top k=1 vs Random dithering 1-bit vs No compression for the error compensated L-Katyusha on \texttt{a5a} and  \texttt{mushrooms} datasets.}
	\label{fig:iter}
\end{figure*}

\begin{figure*}[t]
	\vspace{0cm}
	\centering
	\begin{tabular}{cccc}
		\includegraphics[width=0.3\linewidth]{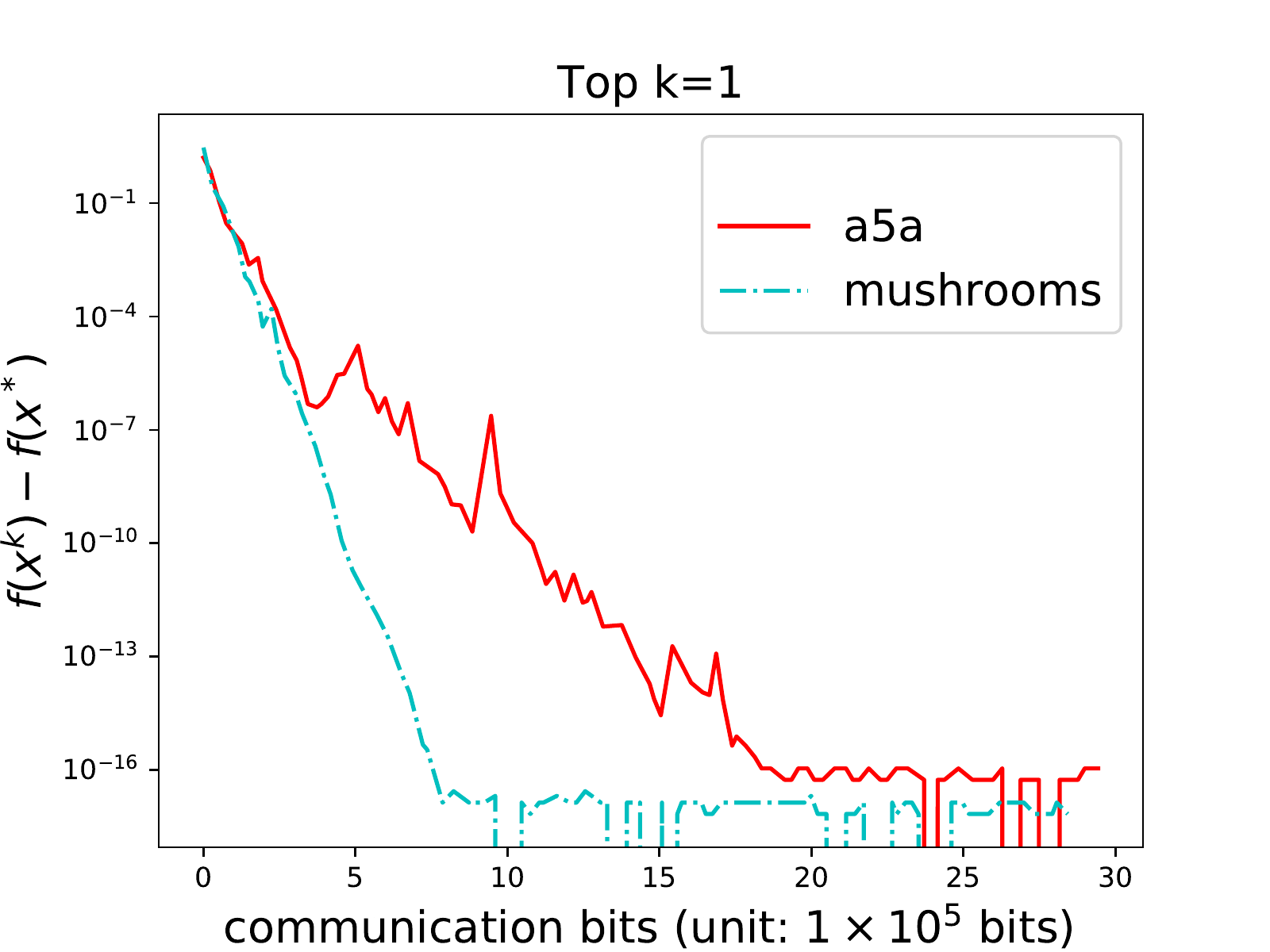}&
		\includegraphics[width=0.3\linewidth]{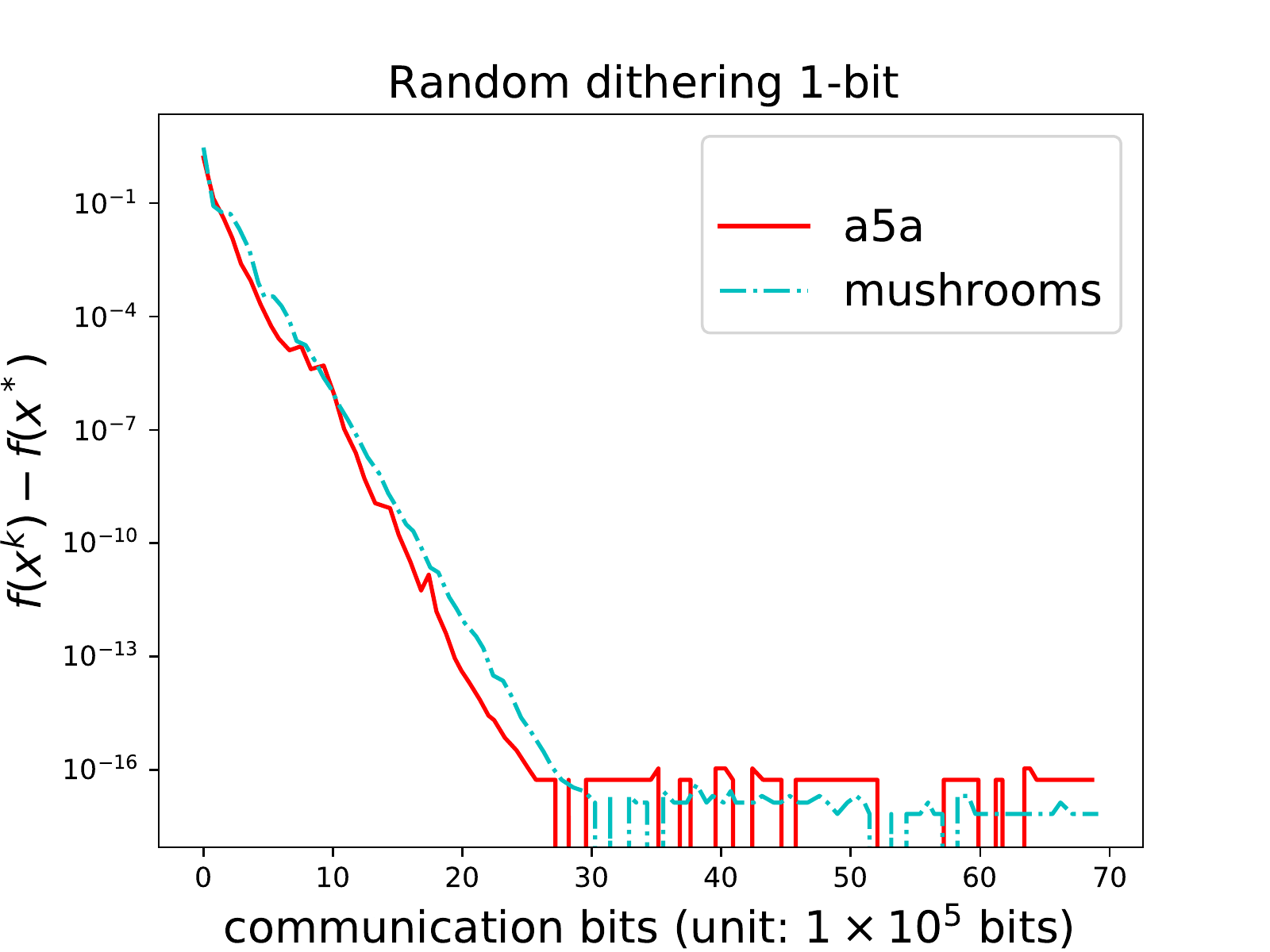}&
		\includegraphics[width=0.3\linewidth]{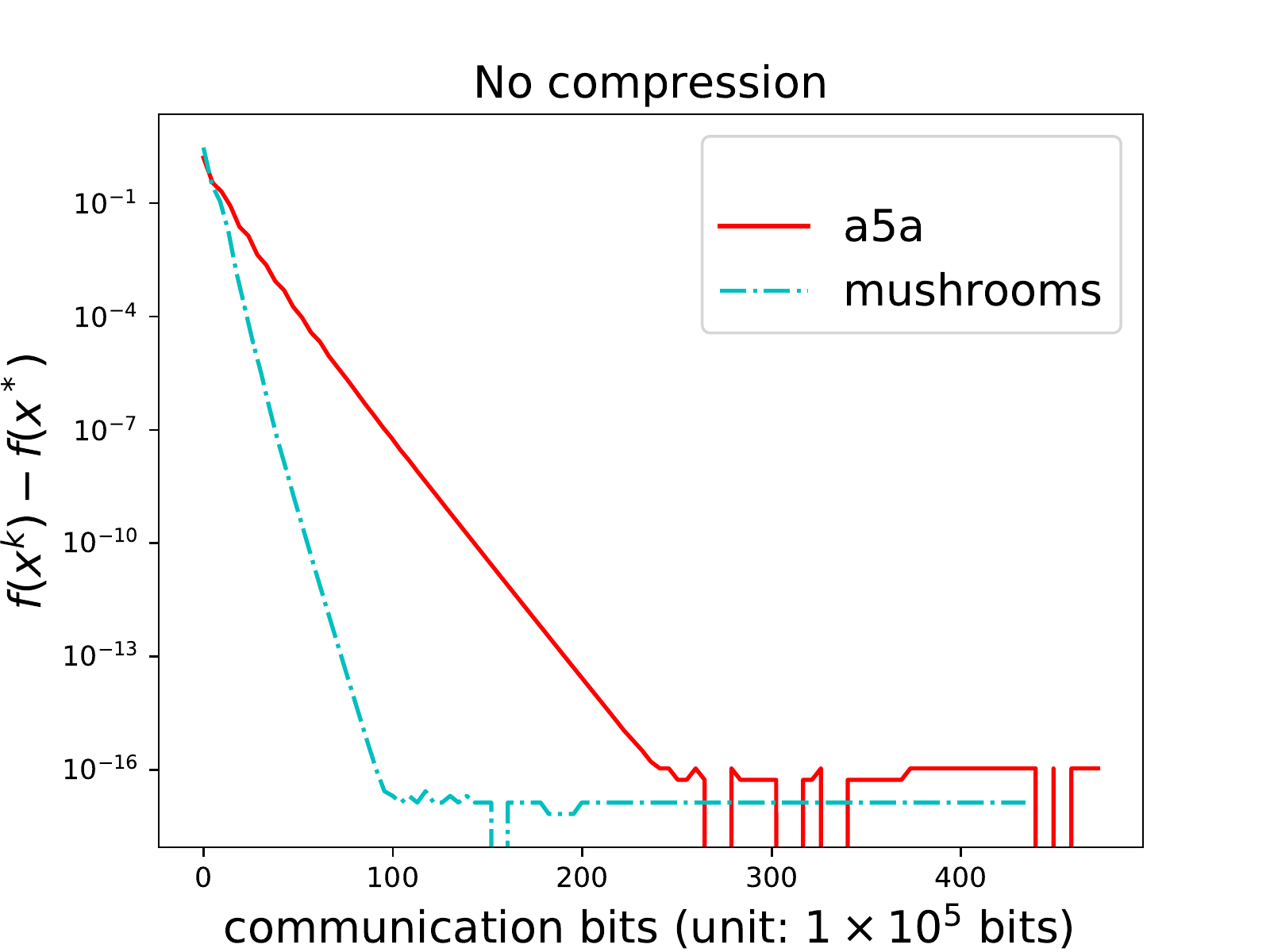}&
		
	\end{tabular}
	\vskip -0.2cm
	\caption{The communication complexity performance of TopK (with $K=1$) vs Random dithering 1-bit vs No compression for the error compensated L-Katyusha on \texttt{a5a} and  \texttt{mushrooms} datasets.}
	\label{fig:combit}
\end{figure*}

\section{Experiments}

In this section, we experimentally study the performance of  error compensated L-Katyusha used with several contraction compressors on the logistic regression problem for binary classification: 
$$
x \mapsto \log\left(  1 + \exp(-y_iA^T_i x)  \right) + \frac{\lambda}{2} \|x\|^2, 
$$
where $\{  A_i, y_i  \}$ is the data point. We use two datasets, namely, $a5a$ and $mushrooms$ from the LIBSVM library \cite{chang2011libsvm}. The regularization parameter $\lambda=10^{-3}$. The number of nodes in our experiments is $20$, and the optimal solution is obtained by running the uncompressed L-Katyusha for $10^5$ iterations. We use the parameter setting in Corollary \ref{co:eclkatyusha} (ii). We calculate the theoretical $L_f$ and $L$ as $L_f^{th}$ and $L^{th}$ respectively. Then we choose $L_f = t \cdot L_f^{th}$ and $L=t \cdot L^{th}$, and search the best $t$ for $t \in \{ 10^{-k} | k=0, 1, 2...  \}$ in each case. 

\paragraph{Compressors.} In our experiments, we use two contraction compressors: TopK compressor with $K=1$ and compressor $\frac{1}{\omega + 1} {\tilde Q}$, where ${\tilde Q}$ is the unbiased random dithering compressor in \cite{Alistarh17} with level $s=2^1$. For TopK compressor, $r(Q) = \frac{K(64 + \lceil \log d \rceil)}{64d}$. For random dithering compressor, from Theorem~3.2 in \cite{Alistarh17}, we can get 
$$
\mytextstyle 
r(Q) = \frac{1}{64 d} \left(  \left(  3 + \left(  \frac{3}{2} + o(1)  \right) \log \left(  \frac{2(s^2 + d)}{s(s + \sqrt{d})}  \right)  \right)s(s + \sqrt{d})  + 64 \right). 
$$

\subsection{TopK vs Random dithering vs No compression}

In this subsection, we compare the uncompressed L-Katyusha with the error compensated L-Katyusha with two contraction compressors: TopK compressor and random dithering compressor. For simplicity, we choose $p=r(Q)$, and explore the influence of $p$ in the next subsection. Figure \ref{fig:iter} and figure \ref{fig:combit} show the iteration complexity and communication complexity of them respectively. We can see that compared with the uncompressed L-Katyusha, the error compensated L-Katyusha with TopK and random dithering compressors need more iterations to reach the optimal solution, but need much less communication bits. In particular, the error compensated L-Katyusha with Top1 compressor is more than 10 times faster than the umcompressed L-Katyusha considering the communication complexity.

\begin{figure*}[t]
	\centering
	\begin{tabular}{ccc}
		\includegraphics[width=0.45\linewidth]{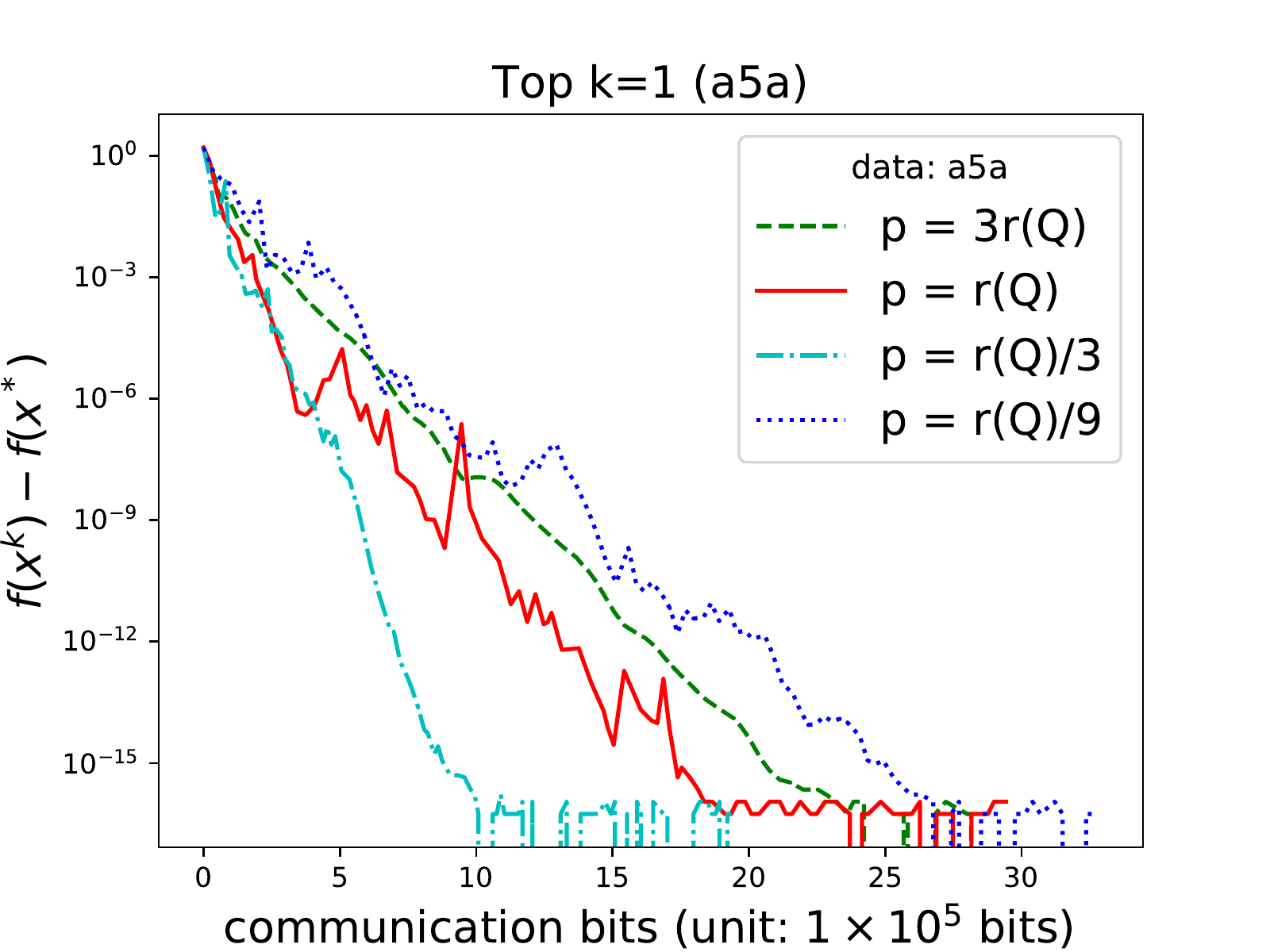}&
		\includegraphics[width=0.45\linewidth]{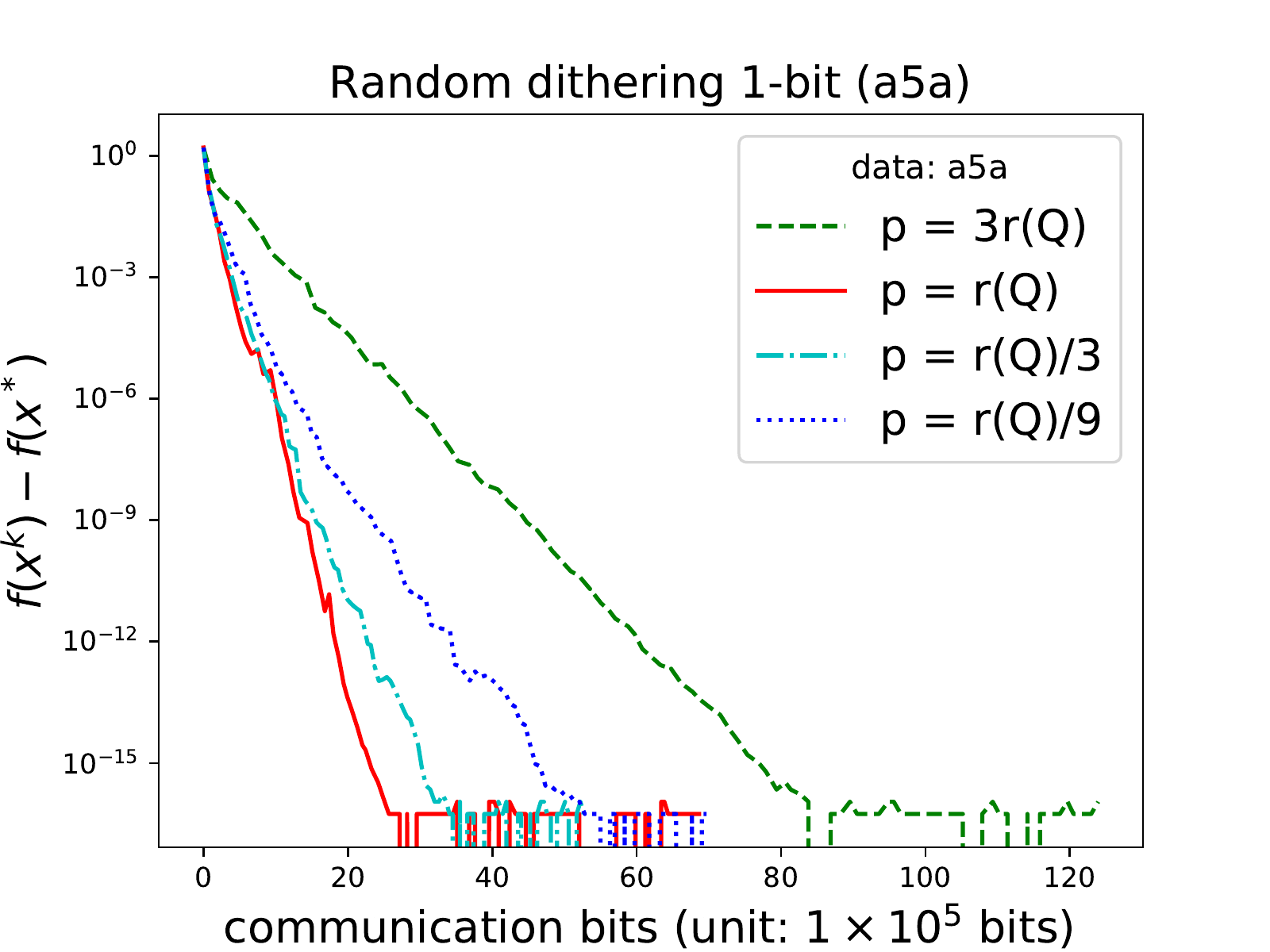} \\
		\includegraphics[width=0.45\linewidth]{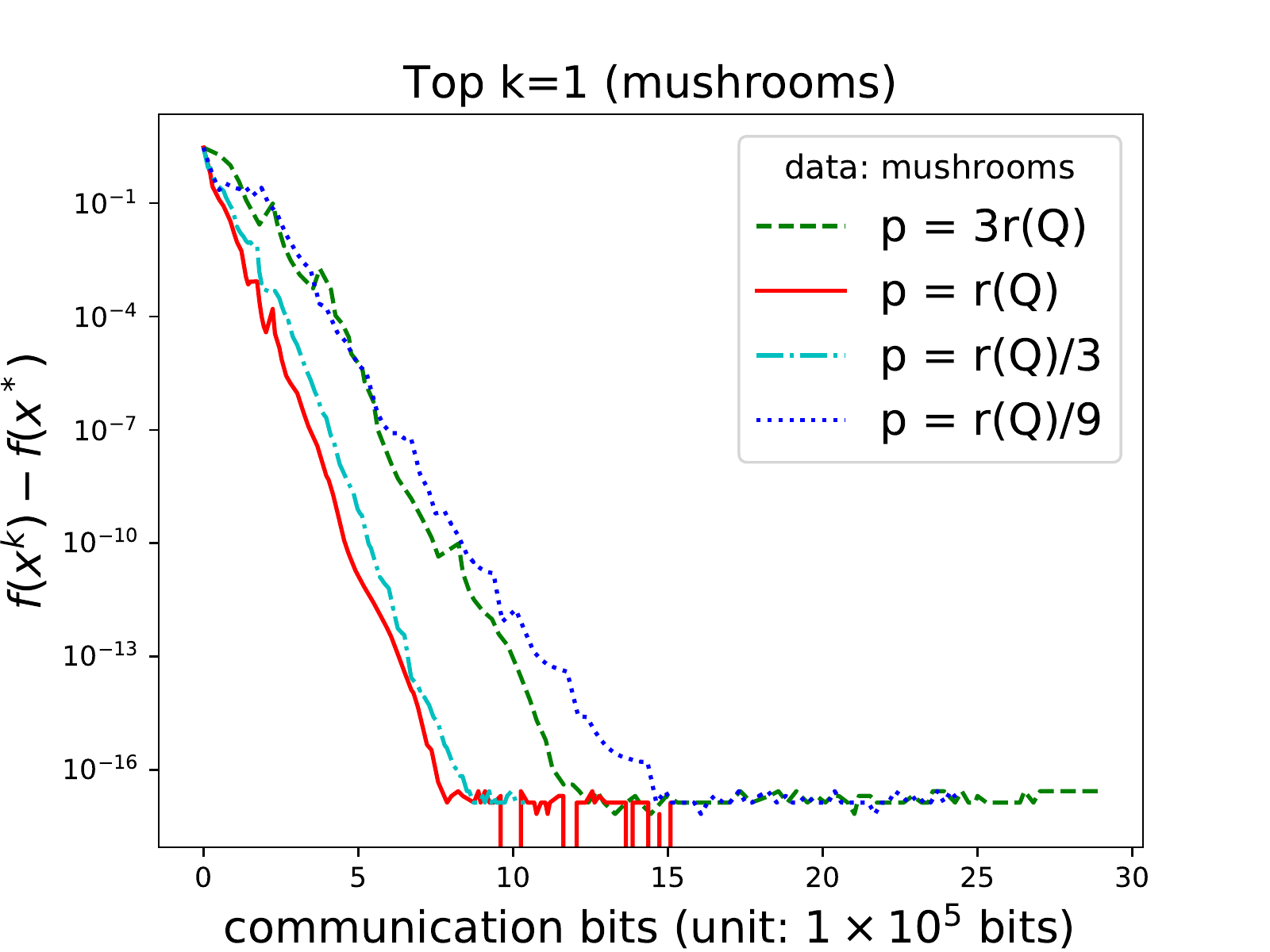}&
		\includegraphics[width=0.45\linewidth]{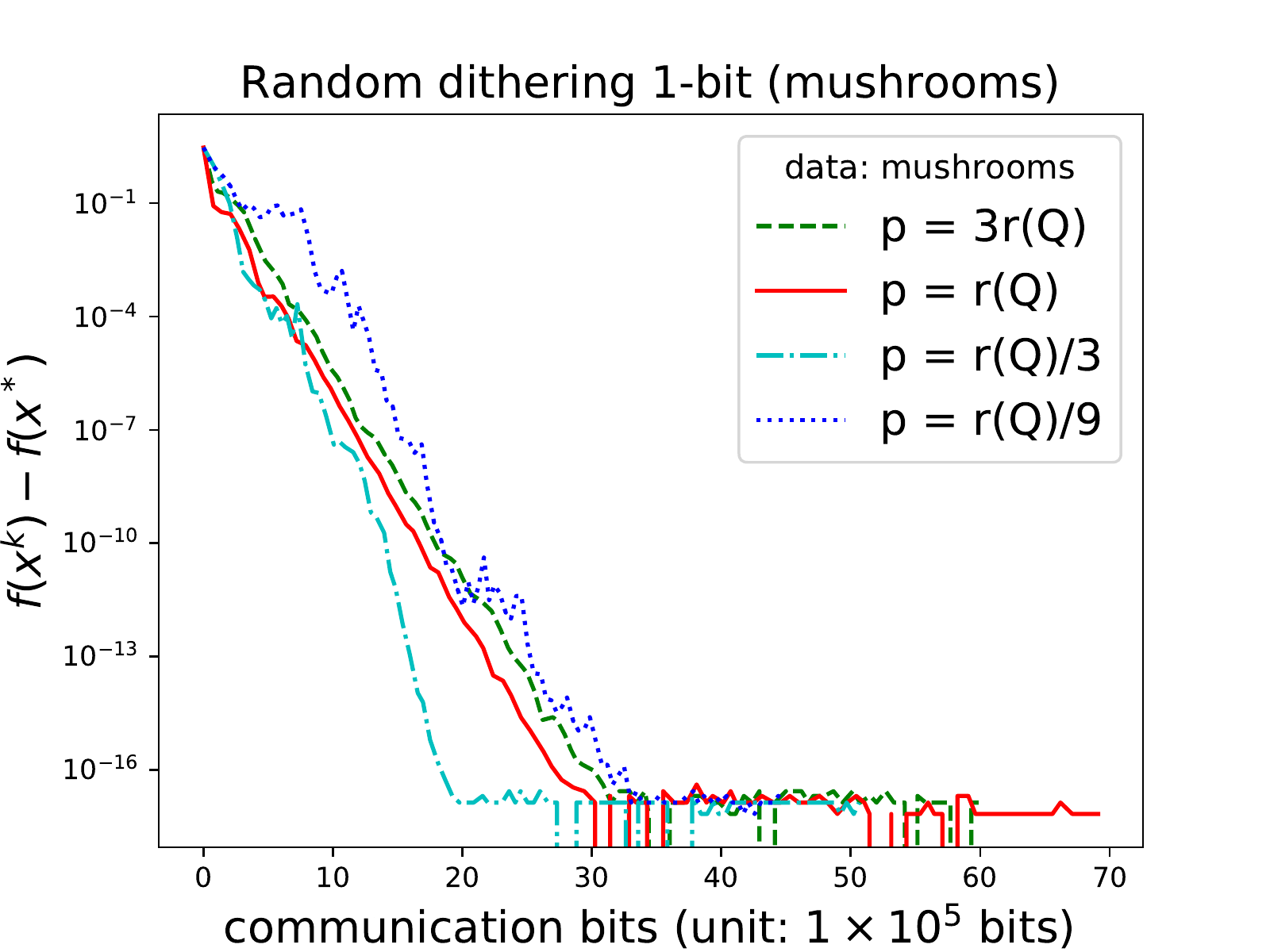}&
	\end{tabular}
	\caption{The influence of $p$ for the communication complexity performance of Top k=1 and Random dithering 1-bit for the error compensated L-Katyusha on \texttt{a5a} and  \texttt{mushrooms} datasets.}
	\label{fig:p}
\end{figure*}

\begin{figure*}[t]
	\vspace{0cm}
	\centering
	\begin{tabular}{cccc}
		\includegraphics[width=0.48\linewidth]{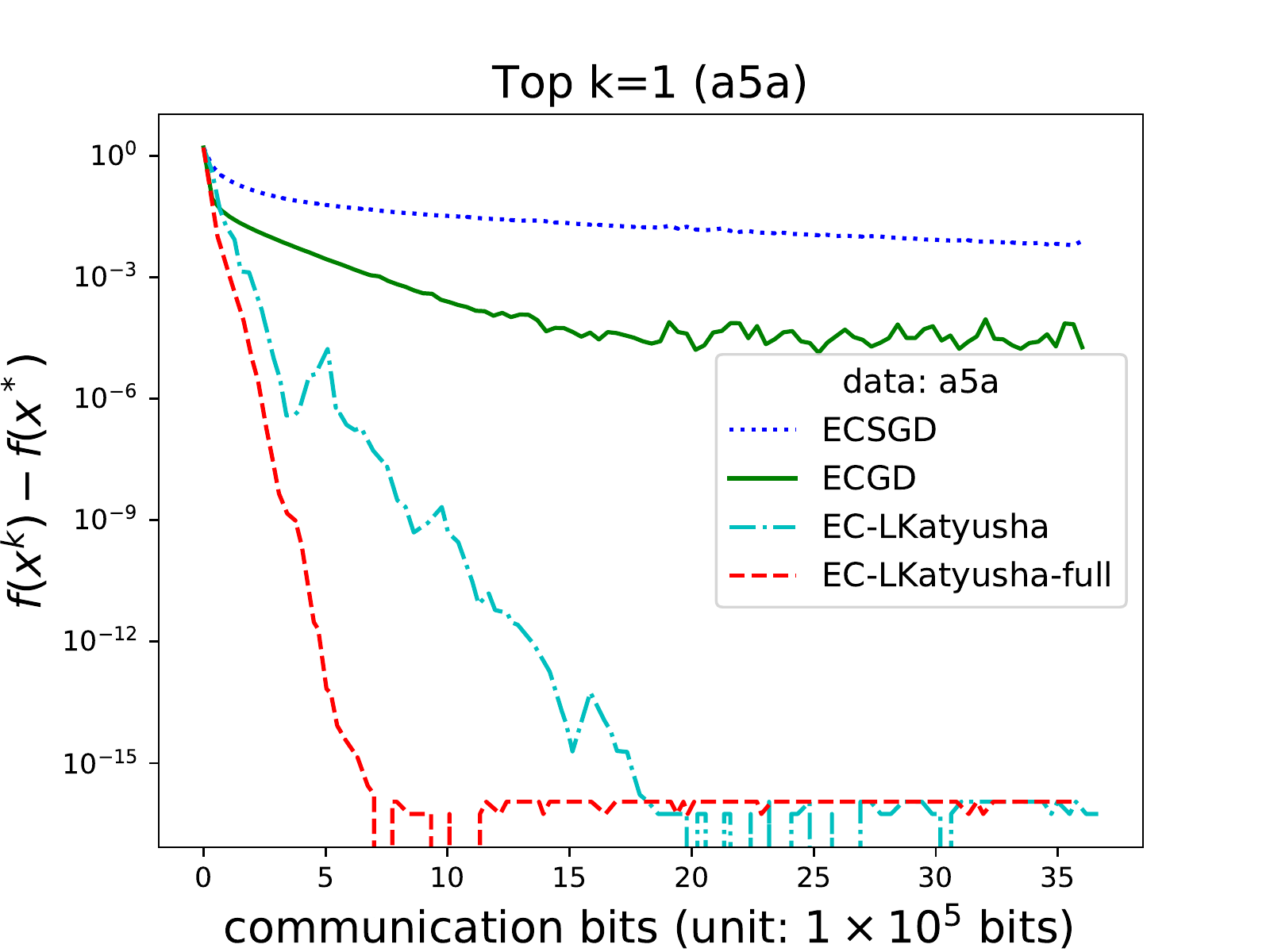}&
		\includegraphics[width=0.48\linewidth]{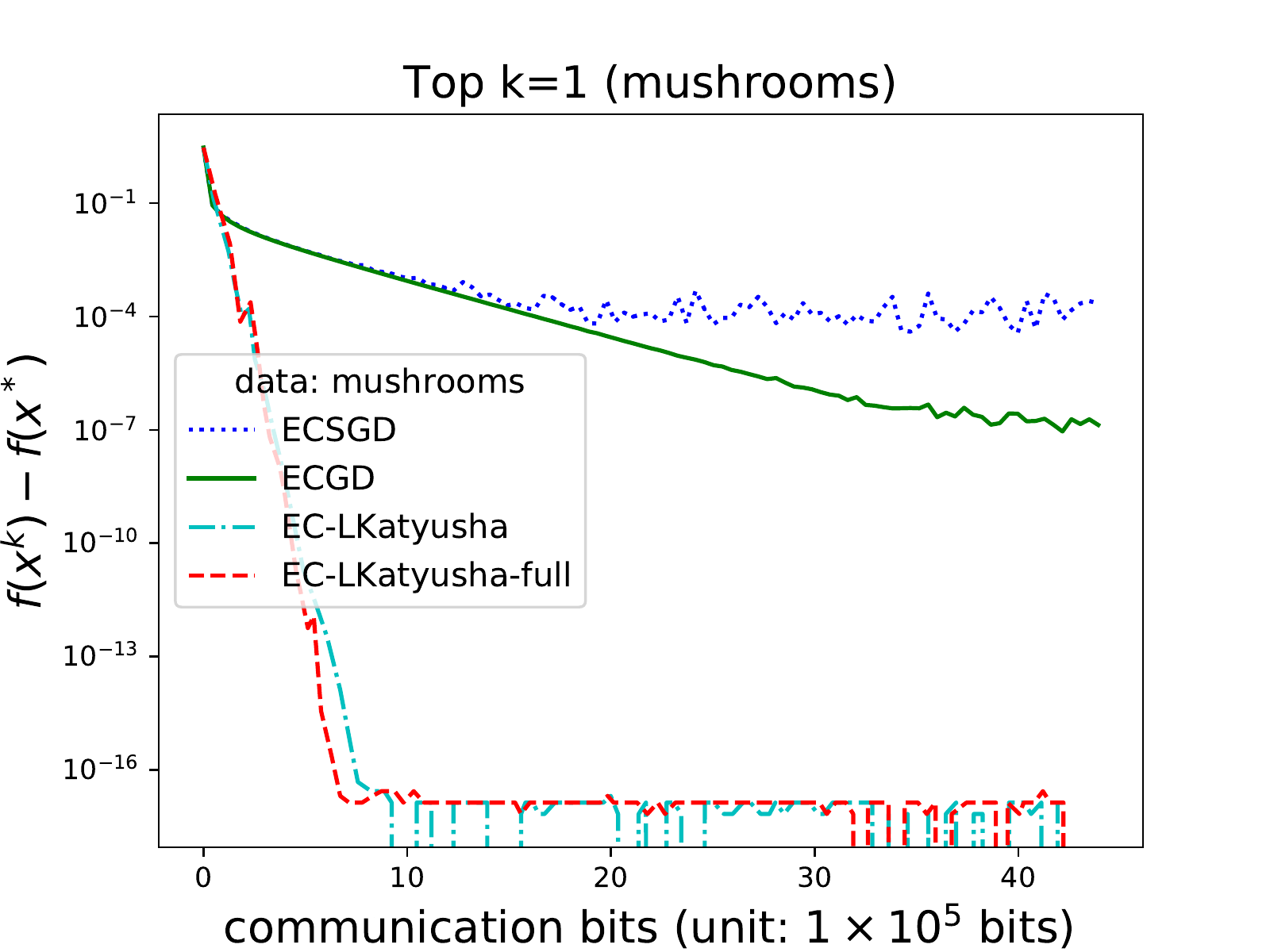}&
		
	\end{tabular}
	\vskip -0.2cm
	\caption{The communication complexity performance of ECSGD vs ECGD vs EC-LKatyusha vs EC-LKatyusha-full for Top k=1 on \texttt{a5a} and  \texttt{mushrooms} datasets.}
	\label{fig:topkcompare}
\end{figure*}

\begin{figure*}[t]
	\vspace{0cm}
	\centering
	\begin{tabular}{cccc}
		\includegraphics[width=0.48\linewidth]{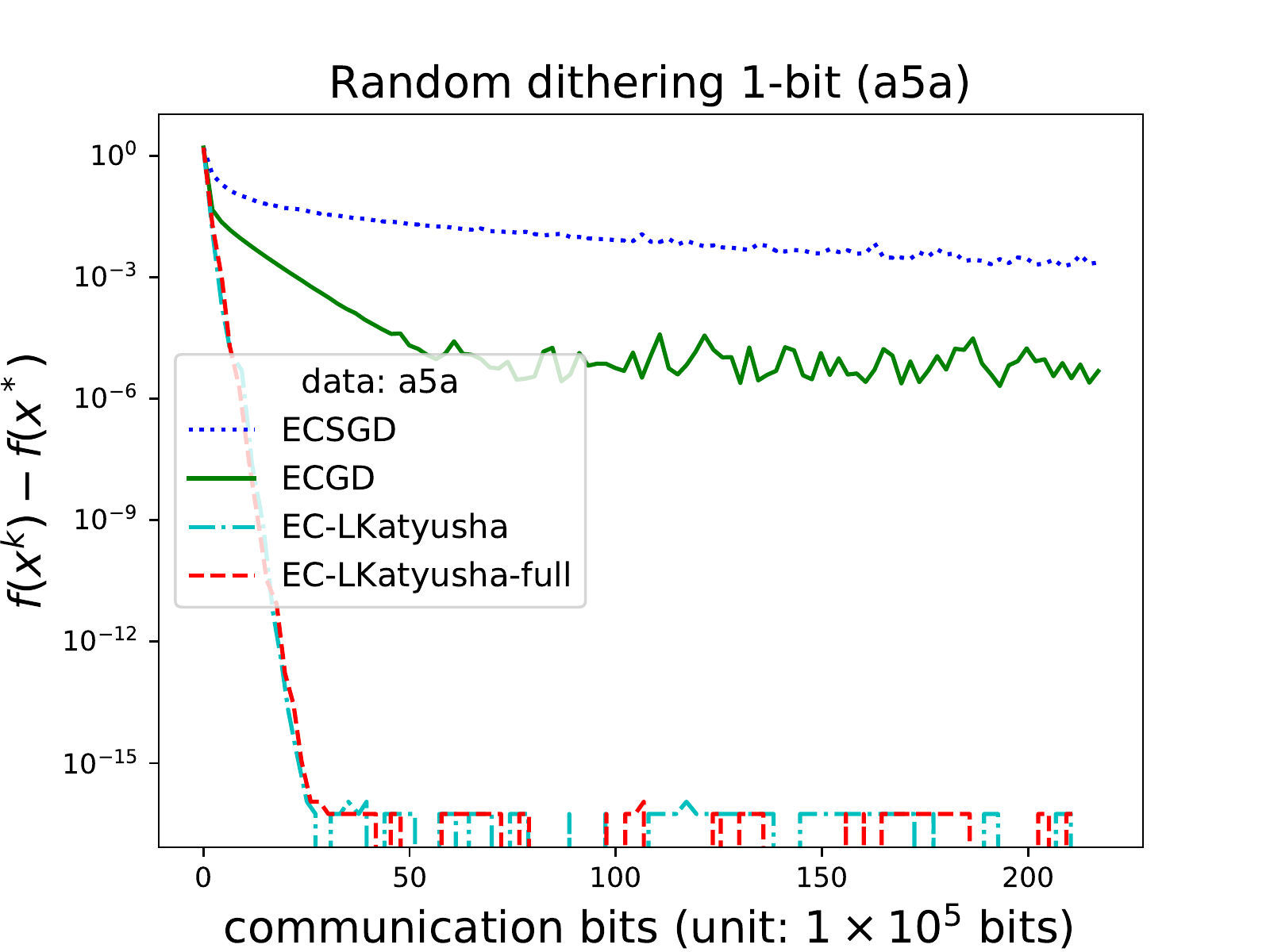}&
		\includegraphics[width=0.48\linewidth]{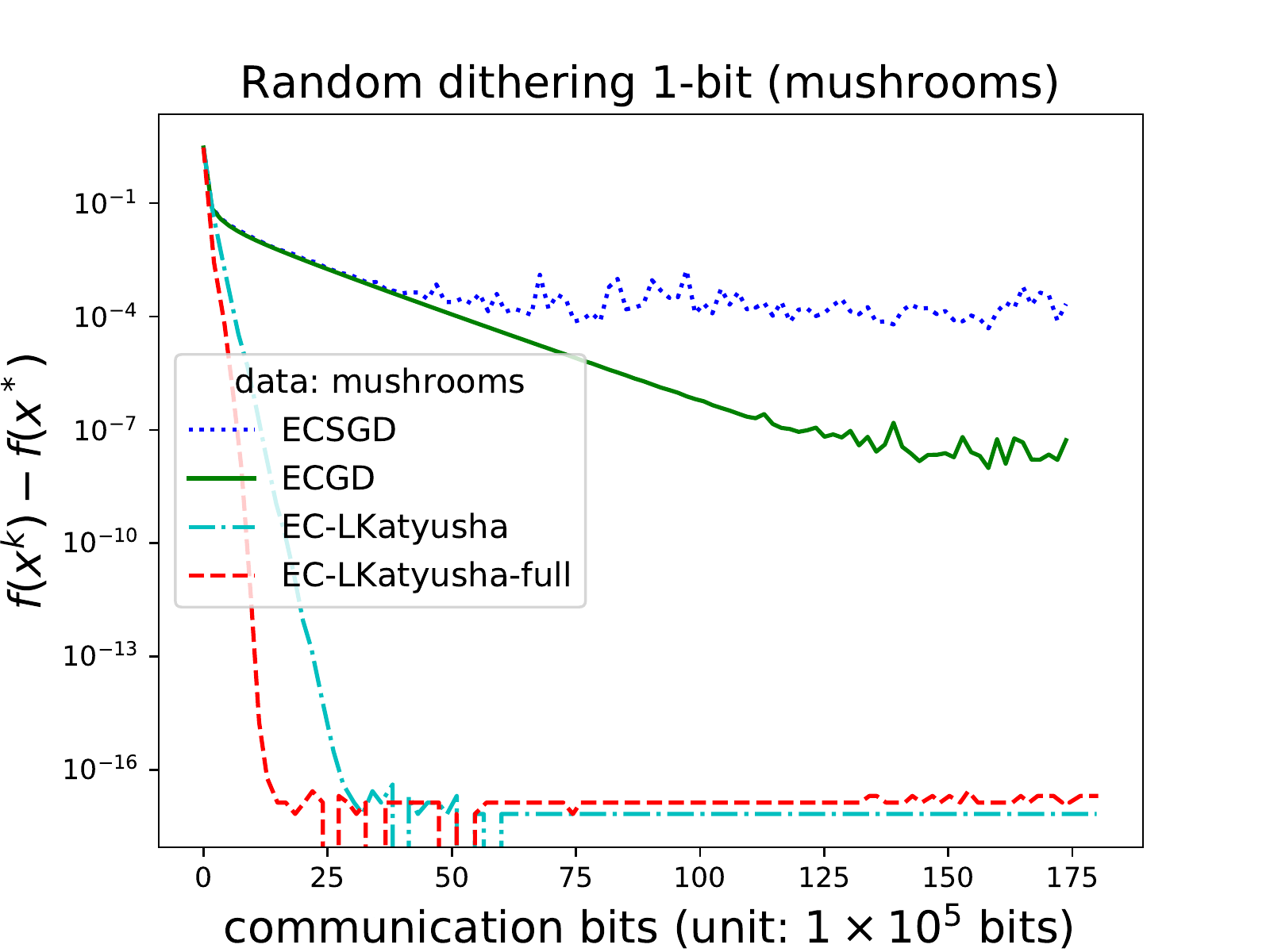}&
		
	\end{tabular}
	\vskip -0.2cm
	\caption{The communication complexity performance of ECSGD vs ECGD vs EC-LKatyusha vs EC-LKatyusha-full for Random dithering 1-bit on \texttt{a5a} and  \texttt{mushrooms} datasets.}
	\label{fig:randomcompare}
\end{figure*}

\begin{figure*}[t]
	\vspace{0cm}
	\centering
	\begin{tabular}{cccc}
		\includegraphics[width=0.48\linewidth]{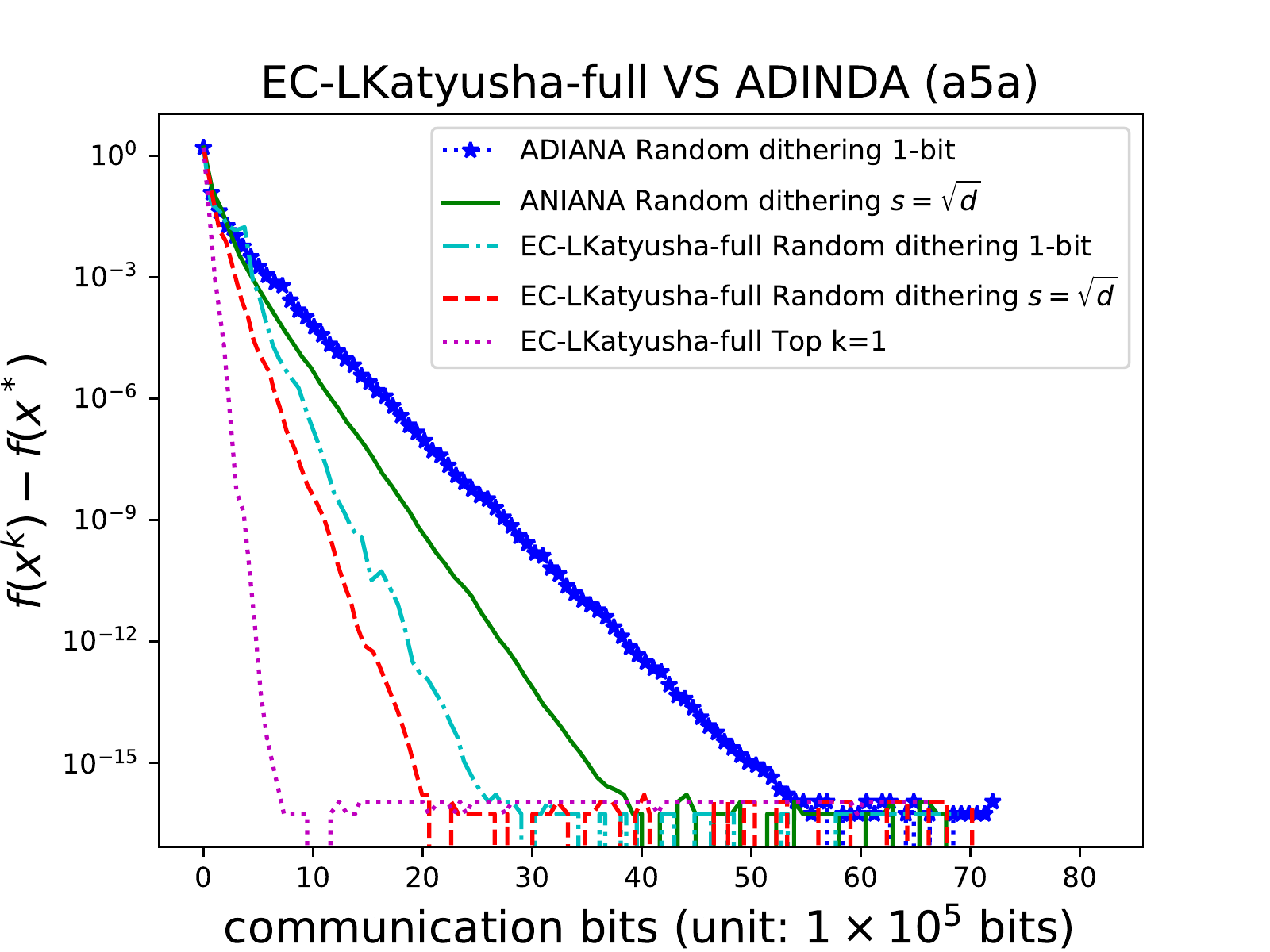}&
		\includegraphics[width=0.48\linewidth]{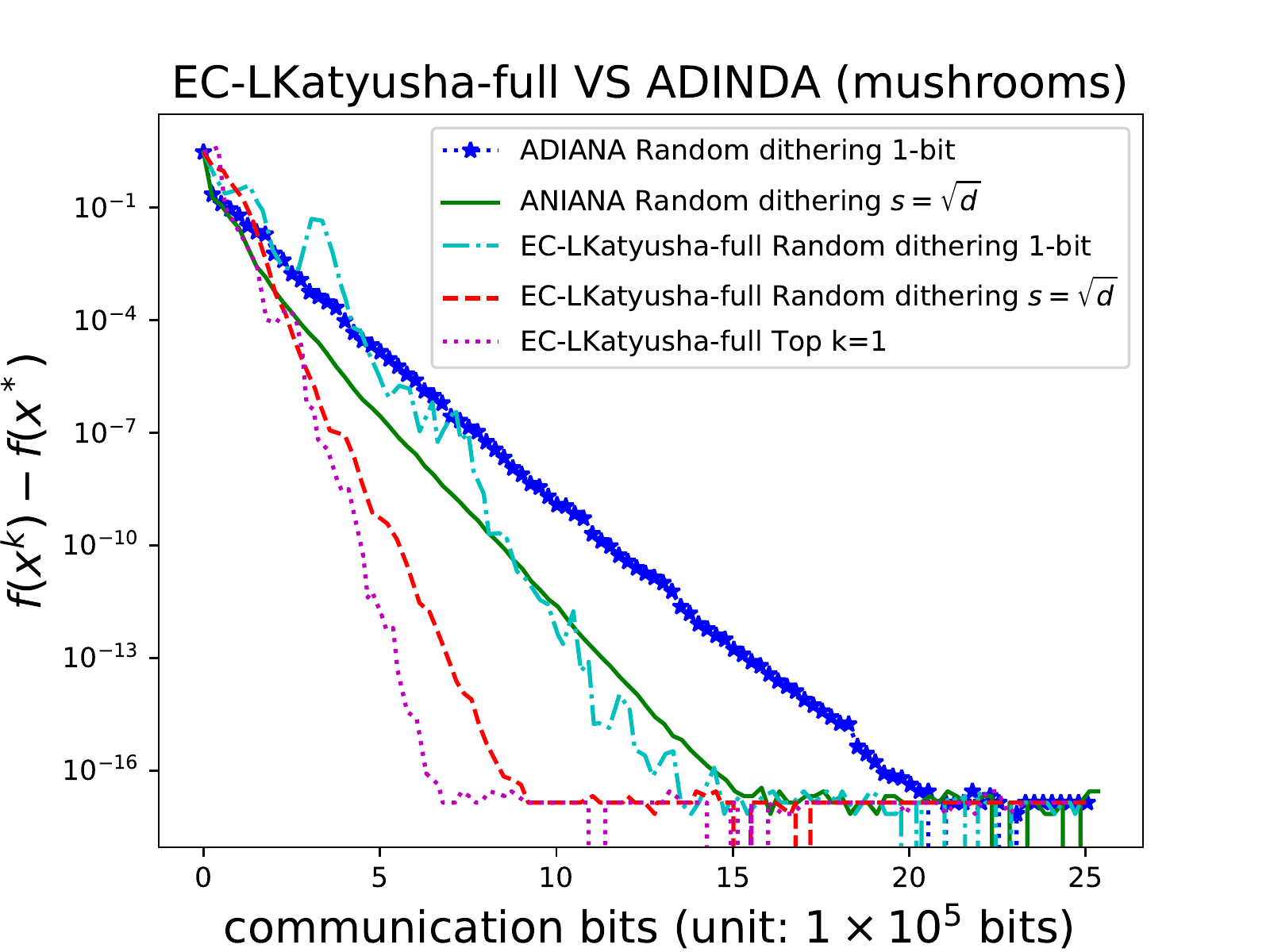}&
		
	\end{tabular}
	\vskip -0.2cm
	\caption{The communication complexity performance of EC-LKatyusha-full vs ADIANA on \texttt{a5a} and  \texttt{mushrooms} datasets.}
	\label{fig:compareANIANA}
\end{figure*}

\subsection{Influence of $p$} 

In this subsection, we show the influence of the parameter $p$ for the communication complexity of the error compensated L-Katyusha with TopK and random dithering compressors respectively. We choose $p = t \cdot r(Q)$ for $t\in \left\{  3, 1, \frac{1}{3}, \frac{1}{9}  \right\}$. Figure \ref{fig:p} shows that $p=r(Q)$ or $p=\frac{1}{3}r(Q)$ achieves the best performance, which coincides with our analysis in Section \ref{sec:com}.

\subsection{Comparison to ECSGD and ECGD}

In this subsection, we compare error compensated L-Katyusha (ECLK) with error compensated SGD (ECSGD) and error compensated GD (ECGD) for TopK compressor and random dithering compressor. ECGD is actually a special case of ECSGD with $m=1$, where the full gradient $\nabla f^{(\tau)}(x^k)$ is calculated on each node. Let ECLK-full be the special case of  ECLK with $m=1$, where the full gradient $\nabla f^{(\tau)}(x^k)$ is calculated on each node. For ECLK, we choose $p=r(Q)$. Figure~\ref{fig:topkcompare} and Figure~\ref{fig:randomcompare} show that ECSGD and ECGD can only converge to a neighborhood of the optimal solution, while ECLK and ECLK-full converge to the optimal solution, and at a linear rate.

\subsection{Comparison to ADIANA}

ADIANA \cite{li2020acceleration} is an accelerated method for any unbiased compressor where the full gradient is used on each node. In this subsection, we compare the EC-LKatyusha-full with ADIANA. For the unbiased compressor ${\tilde Q}$ for ADIANA, we use random dithering compressor with $s=2$ and $s = \sqrt{d}$. For the contraction compressor, we use TopK compressor with $K=1$ and $\frac{1}{\omega + 1}{\tilde Q}$ where ${\tilde Q}$ is the random dithering compressor with $s=2$ and $s = \sqrt{d}$. Figure \ref{fig:compareANIANA} shows that for the communication complexity, the EC-LKatyusha-full with Top1 compressor is the best. For the random dithering compressor with $s=2$ or $s = \sqrt{d}$, the communication complexity of EC-LKatyusha-full is also better than that of ADIANA.

\clearpage

	\bibliographystyle{plain}
	\bibliography{error_compensated_L_Katyusha.bib}

	\clearpage
	\appendix
    \footnotesize

	\part*{Appendix}

	\tableofcontents
	
	\clearpage

	\section{Lemmas}

	We bound the varaince of $g^k$ in the following lemma. 
	
	\begin{lemma}\label{lm:expL2}
		We have 
		\begin{eqnarray}
	 \mathbb{E}_k\left[  \|g^k + \nabla f(w^k) - \nabla f(x^k)\|^2  \right] 
		&\leq & \frac{2L}{n} \left(  f(w^k) - f(x^k) - \langle \nabla f(x^k), w^k - x^k \rangle  \right).  \label{eq:expL2}
		\end{eqnarray}
	\end{lemma}

	\begin{proof}
	
	Since $f_i^{(\tau)}$ is $L$-smooth, we have 
	$$
	\|\nabla f_i^{(\tau)}(x) - \nabla f_i^{(\tau)}(y)\|^2 \leq 2L (f_i^{(\tau)}(x) - f_i^{(\tau)}(y) - \langle \nabla f_i^{(\tau)}(y), x-y\rangle ), 
	$$
	for any $x, y \in \R^d$. Therefore, 
	\begin{eqnarray*}
 \mathbb{E}_k \left[ \|g^k + \nabla f(w^k) - \nabla f(x^k)\|^2 \right] 
		&=& \mathbb{E}_k \|g^k\|^2 - \|\nabla f(x^k) - \nabla f(w^k)\|^2 \\ 
		&=& \mathbb{E}_k \left\| \frac{1}{n} \sum_{\tau=1}^n g^k_\tau \right\|^2 - \|\nabla f(x^k) - \nabla f(w^k)\|^2 \\
		&=& \frac{1}{n^2} \mathbb{E}_k \left \langle \sum_{\tau=1}^n g^k_{\tau},  \sum_{\tau=1}^n g^k_{\tau} \right \rangle - \|\nabla f(x^k) - \nabla f(w^k)\|^2 \\ 
		&=& \frac{1}{n^2} \sum_{\tau_1, \tau_2=1}^n \mathbb{E}_k \left\langle  g^k_{\tau_1}, g^k_{\tau_2} \right\rangle - \|\nabla f(x^k) - \nabla f(w^k)\|^2 \\ 
		&=& \frac{1}{n^2} \sum_{\tau=1}^n \mathbb{E}_k \|g^k_{\tau}\|^2  - \|\nabla f(x^k) - \nabla f(w^k)\|^2 \\
		&&  + \frac{1}{n^2} \sum_{\tau_1 \neq \tau_2} \left \langle \nabla f^{(\tau_1)}(x^k) - \nabla f^{(\tau_1)}(w^k),  \nabla f^{(\tau_2)}(x^k) - \nabla f^{(\tau_2)}(w^k) \right \rangle  \\ 
		&=& \frac{1}{n^2} \sum_{\tau=1}^n \mathbb{E}_k \|g^k_{\tau}\|^2  - \frac{1}{n^2} \sum_{\tau=1}^n \mathbb{E} \|\nabla f^{(\tau)}(x^k) - \nabla f^{(\tau)}(w^k) \|^2 \\ 
		&\leq&  \frac{1}{n^2} \sum_{\tau=1}^n \mathbb{E}_k \|g^k_{\tau}\|^2 \\ 
		&\leq& \frac{2L}{n^2} \sum_{\tau=1}^n \mathbb{E}_k  \left(  f^{(\tau)}_{i^\tau_k} (w^k) - f^{(\tau)}_{i^\tau_k}(x^k) - \langle \nabla f^{(\tau)}_{i^\tau_k} (x^k), w^k - x^k \rangle   \right) \\ 
		&=& \frac{2L}{n} \left(  f(w^k) - f(x^k) - \langle \nabla f(x^k), w^k-x^k  \right). 
	\end{eqnarray*}
	
	\end{proof}

	\begin{lemma}\label{lm:couple2eclkatyusha}
		If ${\cal L}_1 \geq L_f$, then we 
		have 
		\begin{equation}
		 \frac{{\cal L}_1}{4\eta} \|z^{k+1} - z^k\|^2 + \langle g^k + \nabla f(w^k), z^{k+1} - z^k \rangle
		\geq  \frac{1}{\theta_1} \left(  f(y^{k+1}) - f(x^k)  \right)  - \frac{1}{{\cal L}_1 \theta_1} \|g^k + \nabla f(w^k) - \nabla f(x^k)\|^2.   \label{eq:couple2}
		\end{equation}
	\end{lemma}
	
	\begin{proof}
	
	Since $z^{k+1} - z^k = \frac{1}{\theta_1}(y^{k+1} - x^k)$, we have 
	\begin{eqnarray*}
	 \frac{{\cal L}_1}{4\eta} \|z^{k+1} - z^k\|^2 + \langle g^k + \nabla f(w^k), z^{k+1} - z^k \rangle
		&=& \frac{{\cal L}_1}{4\eta \theta_1^2} \|y^{k+1} - x^k\|^2 +  \frac{1}{\theta_1} \langle g^k + \nabla f(w^k), y^{k+1} - x^k \rangle \\ 
		&=& \frac{1}{\theta_1} \langle \nabla f(x^k), y^{k+1} - x^k \rangle + \frac{3{\cal L}_1}{4 \theta_1} \|y^{k+1} - x^k\|^2 \\ 
		&& +  \frac{1}{\theta_1} \langle g^k + \nabla f(w^k) - \nabla f(x^k), y^{k+1} - x^k \rangle \\ 
		&\geq& \frac{1}{\theta_1} \left(  f(y^{k+1}) - f(x^k)  \right) + \left(  \frac{3{\cal L}_1}{4 \theta_1} - \frac{L_f}{2 \theta_1} \right) \|y^{k+1} - x^k\|^2 \\ 
		&& + \frac{1}{\theta_1} \langle g^k + \nabla f(w^k) - \nabla f(x^k), y^{k+1} - x^k \rangle \\ 
		&\geq& \frac{1}{\theta_1} \left(  f(y^{k+1}) - f(x^k)  \right) + \frac{{\cal L}_1}{4 \theta_1} \|y^{k+1} - x^k\|^2 \\ 
		&& + \frac{1}{\theta_1} \langle g^k + \nabla f(w^k) - \nabla f(x^k), y^{k+1} - x^k \rangle \\ 
		&\geq&  \frac{1}{\theta_1} \left(  f(y^{k+1}) - f(x^k)  \right) - \frac{1}{{\cal L}_1 \theta_1} \|g^k + \nabla f(w^k) - \nabla f(x^k)\|^2, 
	\end{eqnarray*}
	where the first inequality comes from $L_f$-smoothness of $f$, and the last inequality comes from Young's inequality. 
	
	\end{proof}

	\begin{lemma}\label{lm:couple1eclkatyusha} 
		We have 
		\begin{eqnarray}
		 \langle g^k + \nabla f(w^k), x^*-z^{k+1} \rangle + \frac{\mu_f}{2} \|x^k -x^*\|^2  
		&\geq&  \frac{{\cal L}_1}{4\eta} \|z^k - z^{k+1}\|^2 + {\tilde {\cal Z}}^{k+1} - \frac{{\cal L}_1 {\tilde {\cal Z}}^k}{{\cal L}_1 + \eta\mu/2}  - \left(   \frac{{\cal L}_1}{2\eta} + \frac{\mu_f}{2}  \right) \|e^k\|^2  \nonumber  \\ 
		&& - \left(   \frac{{\cal L}_1}{2\eta} + \frac{\mu}{2}  \right) \|e^{k+1}\|^2  + \psi(z^{k+1}) - \psi(x^*).  \label{eq:couple1}
		\end{eqnarray}
	\end{lemma}
	
	\begin{proof}
	
	First, from (\ref{eq:tildezk+1}) and $\sigma_1 = \frac{\mu_f}{2{\cal L}_1}$, we have 
	\begin{eqnarray*}
		g^k + \nabla f(w^k) &=& \frac{{\cal L}_1}{\eta} ({\tilde z}^k - {\tilde z}^{k+1}) + {\cal L}_1 \sigma_1 ({\tilde x}^k - {\tilde z}^{k+1}) - \partial \psi(z^{k+1}) \\ 
		&=&  \frac{{\cal L}_1}{\eta} ({\tilde z}^k - {\tilde z}^{k+1}) + \frac{\mu_f}{2} ({\tilde x}^k - {\tilde z}^{k+1}) - \partial \psi(z^{k+1}), 
	\end{eqnarray*}
	which implies that 
	\begin{eqnarray*}
		 \langle g^k + \nabla f(w^k),  z^{k+1} - x^* \rangle 
		&=& \frac{\mu_f}{2} \langle z^{k+1} - x^*, {\tilde x}^k - {\tilde z}^{k+1} \rangle + \frac{{\cal L}_1}{\eta} \langle z^{k+1} - x^*, {\tilde z}^k - {\tilde z}^{k+1} \rangle \\ 
		&& - \langle z^{k+1} - x^*, \partial \psi(z^{k+1}) \rangle \\ 
		&\leq& \frac{\mu_f}{2} \langle z^{k+1} - x^*, {\tilde x}^k - {\tilde z}^{k+1} \rangle + \frac{{\cal L}_1}{\eta} \langle z^{k+1} - x^*, {\tilde z}^k - {\tilde z}^{k+1} \rangle \\ 
		&& + \psi(x^*) - \psi(z^{k+1}) - \frac{\mu_{\psi}}{2} \|z^{k+1} - x^*\|^2 \\ 
		&=& \frac{\mu_f}{2} \langle {\tilde z}^{k+1} - x^*, {\tilde x}^k - {\tilde z}^{k+1} \rangle + \frac{{\cal L}_1}{\eta} \langle {\tilde z}^{k+1} - x^*, {\tilde z}^k - {\tilde z}^{k+1} \rangle \\ 
		&& + \psi(x^*) - \psi(z^{k+1}) - \frac{\mu_{\psi}}{2} \|z^{k+1} - x^*\|^2 \\ 
		&& + \frac{\mu_f}{2} \langle z^{k+1} - {\tilde z}^{k+1}, {\tilde x}^k - {\tilde z}^{k+1} \rangle + \frac{{\cal L}_1}{\eta} \langle z^{k+1} - {\tilde z}^{k+1}, {\tilde z}^k - {\tilde z}^{k+1} \rangle \\ 
		&=& \frac{\mu_f}{4} \left(  \|{\tilde x}^k - x^*\|^2 - \|{\tilde z}^{k+1} - x^*\|^2 - \|{\tilde x}^k - {\tilde z}^{k+1}\|^2  \right) \\ 
		&& + \frac{{\cal L}_1}{2\eta} \left(  \|{\tilde z}^k - x^*\|^2 - \|{\tilde z}^{k+1} - x^*\|^2 - \|{\tilde z}^k - {\tilde z}^{k+1}\|^2  \right) \\ 
		&& + \frac{\mu_f}{4} \left(  \|z^{k+1} - {\tilde z}^{k+1} \|^2 + \|{\tilde x}^k - {\tilde z}^{k+1}\|^2 - \|{\tilde x}^k - z^{k+1}\|^2  \right) \\ 
		&&  + \frac{{\cal L}_1}{2\eta} \left(  \|z^{k+1} - {\tilde z}^{k+1}\|^2 + \|{\tilde z}^k - {\tilde z}^{k+1}\|^2 - \|{\tilde z}^k - z^{k+1}\|^2  \right) \\ 
		&& + \psi(x^*) - \psi(z^{k+1}) - \frac{\mu_{\psi}}{2} \|z^{k+1} - x^*\|^2 \\ 
		&\leq& - \left(  \frac{{\cal L}_1}{2 \eta} + \frac{\mu_f}{4} \right) \| {\tilde z}^{k+1} - x^*\|^2 + \frac{{\cal L}_1}{2\eta} \|{\tilde z}^k - x^*\|^2 + \frac{\mu_f}{4} \|{\tilde x}^k - x^*\|^2 \\ 
		&& + \left(  \frac{{\cal L}_1}{2 \eta} + \frac{\mu_f}{4} \right) \|z^{k+1} - {\tilde z}^{k+1}\|^2 - \frac{{\cal L}_1}{2\eta} \|{\tilde z}^k - z^{k+1}\|^2 \\ 
		&& + \psi(x^*) - \psi(z^{k+1}) - \frac{\mu_{\psi}}{2} \|z^{k+1} - x^*\|^2. 
	\end{eqnarray*}
	
	For $\|{\tilde x}^k - x^*\|^2$, $\|{\tilde z}^k - z^{k+1}\|^2$, and $\|z^{k+1} - x^*\|^2$, from Young's inequality, we have 
	$$
	\|{\tilde x}^k - x^*\|^2 \leq 2\|{\tilde x}^k - x^k\|^2 + 2\|x^k - x^*\|^2, \qquad 	\|{\tilde z}^k - z^{k+1}\|^2 \geq \frac{1}{2} \|z^k - z^{k+1}\|^2 - \|z^k - {\tilde z}^k\|^2, 
	$$	
	and
	$$
	\|z^{k+1} - x^*\|^2 \geq \frac{1}{2}\|{\tilde z}^{k+1} - x^*\|^2 - \|z^{k+1} - {\tilde z}^{k+1}\|^2. 
	$$
	
	Hence, we arrive at 
	\begin{eqnarray*}
	 \langle g^k + \nabla f(w^k),  z^{k+1} - x^* \rangle 
		&\leq& - \left(  \frac{{\cal L}_1}{2 \eta} + \frac{\mu_f}{4} + \frac{\mu_{\psi}}{4} \right) \| {\tilde z}^{k+1} - x^*\|^2 + \frac{{\cal L}_1}{2\eta} \|{\tilde z}^k - x^*\|^2 + \frac{\mu_f}{2} \|x^k - x^*\|^2 \\ 
		&& + \frac{\mu_f}{2} \|{\tilde x}^k - x^k\|^2 + \left(  \frac{{\cal L}_1}{2 \eta} + \frac{\mu_f}{4} + \frac{\mu_\psi}{2} \right) \|z^{k+1} - {\tilde z}^{k+1}\|^2 \\ 
		&& - \frac{{\cal L}_1}{4\eta} \|z^k - z^{k+1}\|^2 + \frac{{\cal L}_1}{2\eta} \|z^k - {\tilde z}^k\|^2 + \psi(x^*) - \psi(z^{k+1}) \\ 
		&=&  - \left(  \frac{{\cal L}_1}{2 \eta} + \frac{\mu}{4}  \right) \| {\tilde z}^{k+1} - x^*\|^2 + \frac{{\cal L}_1}{2\eta} \|{\tilde z}^k - x^*\|^2 + \frac{\mu_f}{2} \|x^k - x^*\|^2 \\ 
		&& + \left(   \frac{{\cal L}_1}{2\eta} + \frac{\mu_f}{2}  \right) \|e^k\|^2 + \left(   \frac{{\cal L}_1}{2\eta}  + \frac{\mu}{2}  \right) \|e^{k+1}\|^2 \\ 
		&& - \frac{{\cal L}_1}{4\eta} \|z^k - z^{k+1}\|^2 + \psi(x^*) - \psi(z^{k+1})\\ 
		&=& - {\tilde {\cal Z}}^{k+1} + \frac{{\cal L}_1 {\tilde {\cal Z}}^k}{{\cal L}_1 + \eta\mu/2} + \frac{\mu_f}{2} \|x^k - x^*\|^2 + \left(   \frac{{\cal L}_1}{2\eta} + \frac{\mu_f}{2}  \right) \|e^k\|^2 \\ 
		&& + \left(   \frac{{\cal L}_1}{2\eta} + \frac{\mu}{2}  \right) \|e^{k+1}\|^2 - \frac{{\cal L}_1}{4\eta} \|z^k - z^{k+1}\|^2 + \psi(x^*) - \psi(z^{k+1}). 
	\end{eqnarray*}
	
	\end{proof}

	\section{Proofs of Lemmas \ref{lm:zyk+1}, \ref{lm:ek+1}, and \ref{lm:ek+1-2} }
	
	\subsection{Proof of Lemma \ref{lm:zyk+1}}
	
	Since $\theta_1 + \theta_2 \leq 1$, and $f$ is $\mu_f$-strong convex, we have 
	\begin{eqnarray*}
		f(x^*) &{\geq}& f(x^k) + \langle \nabla f(x^k), x^* - x^k \rangle + \frac{\mu_f}{2}\|x^k - x^*\|^2 \\ 
		&=& f(x^k) + \frac{\mu_f}{2}\|x^k-x^*\|^2 + \langle \nabla f(x^k), x^*-z^k + z^k -x^k \rangle \\ 
		&=& f(x^k) + \frac{\mu_f}{2}\|x^k -x^*\|^2 + \langle \nabla f(x^k), x^*-z^k \rangle + \frac{\theta_2}{\theta_1}\langle \nabla f(x^k), x^k-w^k \rangle  + \frac{1-\theta_1-\theta_2}{\theta_1} \langle \nabla f(x^k), x^k-y^k \rangle \\ 
		&{\geq}& f(x^k) + \frac{\theta_2}{\theta_1}\langle \nabla f(x^k), x^k-w^k \rangle + \frac{1-\theta_1-\theta_2}{\theta_1} (f(x^k) - f(y^k)) \\ 
		&& + \mathbb{E}_k \left[ \frac{\mu_f }{2}\|x^k - x^*\|^2 + \langle g^k + \nabla f(w^k), x^* - z^{k+1} \rangle + \langle g^k + \nabla f(w^k), z^{k+1}-z^k \rangle \right],
	\end{eqnarray*}
	where the last inequality follows from the convexity of $f$ and $\mathbb{E}_k[g^k + \nabla f(w^k)] = \nabla f(x^k)$. For the last term in the above inequality,   we have 
	\begin{eqnarray*}
	&&	 \mathbb{E}_k \left[ \frac{\mu_f}{2}\|x^k - x^*\|^2 + \langle g^k + \nabla f(w^k), x^* - z^{k+1} \rangle   + \langle g^k + \nabla f(w^k), z^{k+1}-z^k\rangle  - \psi(z^{k+1}) + \psi(x^*) - {\tilde {\cal Z}}^{k+1}\right]   \\  
		&\overset{(\ref{eq:couple1})}{\geq}& - \frac{{\cal L}_1{\tilde {\cal Z}}^k}{{\cal L}_1 + \eta\mu/2} + \mathbb{E}_k\left[ \langle g^k+ \nabla f(w^k), z^{k+1} -z^k\rangle + \frac{{\cal L}_1}{4\eta}\|z^k - z^{k+1}\|^2 \right] \\ 
		&& - \left(   \frac{{\cal L}_1}{2\eta} + \frac{\mu_f}{2}  \right)  \|e^k\|^2  - \left(   \frac{{\cal L}_1}{2\eta} + \frac{\mu}{2}  \right) \mathbb{E}_k \|e^{k+1}\|^2 \\ 
		&\overset{(\ref{eq:couple2})}{\geq}& - \frac{{\cal L}_1{\tilde {\cal Z}}^k}{{\cal L}_1 + \eta\mu/2} - \left(   \frac{{\cal L}_1}{2\eta} + \frac{\mu_f}{2}  \right)  \|e^k\|^2  - \left(   \frac{{\cal L}_1}{2\eta} + \frac{\mu}{2}  \right) \mathbb{E}_k \|e^{k+1}\|^2 \\ 
		&&  + \mathbb{E}_k \left[ \frac{1}{\theta_1}(f(y^{k+1}) - f(x^k)) - \frac{1}{{\cal L}_1\theta_1}\|g^k + \nabla f(w^k) - \nabla f(x^k)\|^2 \right]  \\ 
		&\overset{(\ref{eq:expL2})}{\geq}& - \frac{{\cal L}_1{\tilde {\cal Z}}^k}{{\cal L}_1 + \eta\mu/2} - \left(   \frac{{\cal L}_1}{2\eta} + \frac{\mu_f}{2}  \right)  \|e^k\|^2  - \left(   \frac{{\cal L}_1}{2\eta} + \frac{\mu}{2}  \right) \mathbb{E}_k \|e^{k+1}\|^2 \\ 
		&&  + \mathbb{E}_k \left[ \frac{1}{\theta_1}(f(y^{k+1}) - f(x^k)) - \frac{2L}{n{\cal L}_1\theta_1}(f(w^k) - f(x^k) - \langle \nabla f(x^k), w^k-x^k \rangle ) \right]. 
	\end{eqnarray*}
	
	Therefore, 	
	\begin{eqnarray*}
		&& \mathbb{E}_k \left[ f(x^*) - \psi(z^{k+1}) + \psi(x^*) - {\tilde {\cal Z}}^{k+1} \right] +  \left(   \frac{{\cal L}_1}{2\eta} + \frac{\mu_f}{2}  \right)  \|e^k\|^2  + \left(   \frac{{\cal L}_1}{2\eta} + \frac{\mu}{2}  \right) \mathbb{E}_k \|e^{k+1}\|^2 \\
		&\geq& - \frac{{\cal L}_1{\tilde {\cal Z}}^k}{{\cal L}_1 + \eta\mu/2} - \frac{1-\theta_1-\theta_2}{\theta_1} f(y^k) + \frac{1}{\theta_1} \mathbb{E}_k [f(y^{k+1})]  - \frac{\theta_2}{\theta_1} \left(  f(x^k) + \langle \nabla f(x^k), w^k - x^k \rangle  \right) \\ 
		&& - \frac{2L}{n{\cal L}_1\theta_1}(f(w^k) - f(x^k) - \langle \nabla f(x^k), w^k-x^k \rangle ) \\
		&=&  - \frac{{\cal L}_1{\tilde {\cal Z}}^k}{{\cal L}_1 + \eta\mu/2} - \frac{1-\theta_1-\theta_2}{\theta_1} f(y^k) + \frac{1}{\theta_1}\mathbb{E}_k [f(y^{k+1})] - \frac{\theta_2}{\theta_1} f(w^k) \\ 
		&& + \frac{1}{\theta_1} \left(  \theta_2 - \frac{2L}{n{\cal L}_1}  \right) (f(w^k) - f(x^k) - \langle \nabla f(x^k), w^k-x^k \rangle ). 
	\end{eqnarray*}
	
	From the convexity of $\psi$, and 
	$$
	y^{k+1} = x^k + \theta_1 (z^{k+1} - z^k) = \theta_1 z^{k+1} + \theta_2 w^k + (1- \theta_1 - \theta_2) y^k, 
	$$
	we have 
	$$
	\psi(z^{k+1}) \geq \frac{1}{\theta_1}\psi(y^{k+1}) - \frac{\theta_2}{\theta_1}\psi(w^k) - \frac{1-\theta_1 -\theta_2}{\theta_1} \psi(y^k). 
	$$
	
	Hence, we can obtain 
	\begin{eqnarray*}
		 P(x^*) +  \left(   \frac{{\cal L}_1}{2\eta} + \frac{\mu_f}{2}  \right)  \|e^k\|^2  + \left(   \frac{{\cal L}_1}{2\eta} + \frac{\mu}{2}  \right) \mathbb{E}_k \|e^{k+1}\|^2
		&\geq & \mathbb{E}_k[{\tilde {\cal Z}}^{k+1}] - \frac{{\cal L}_1{\tilde {\cal Z}}^k}{{\cal L}_1 + \eta\mu/2} - \frac{1-\theta_1 - \theta_2}{\theta_1} P(y^k)  \\
		&& \quad + \frac{1}{\theta_1} \mathbb{E}_k [P(y^{k+1})] - \frac{\theta_2}{\theta_1} P(w^k) \\
		&& \quad + \frac{1}{\theta_1} \left(  \theta_2 - \frac{2L}{n{\cal L}_1}  \right) (f(w^k) - f(x^k) - \langle \nabla f(x^k), w^k-x^k \rangle ). 
	\end{eqnarray*}
	
	After rearranging we can get the result.

	\subsection{Proof of Lemma \ref{lm:ek+1}}

	First, we have 
	\begin{eqnarray*}
		 \mathbb{E}_k [\|e^{k+1}_\tau\|^2 ] 
		&=& \mathbb{E}_k \left\|e^k_\tau + \frac{\eta}{{\cal L}_1} g^k_\tau - {\tilde g}^k_\tau \right\|^2 \\ 
		&\leq& (1-\delta) \mathbb{E}_k \left\|e^k_\tau + \frac{\eta}{{\cal L}_1} g^k_\tau \right\|^2 \\ 
		&=& (1-\delta) \left\|e^k_\tau + \frac{\eta}{{\cal L}_1}(\nabla f^{(\tau)}(x^k) - \nabla f^{(\tau)}(w^k) ) \right\|^2  + \frac{\eta^2}{{\cal L}_1^2}(1-\delta) \mathbb{E}_k\|g^k_\tau - (\nabla f^{(\tau)}(x^k) - \nabla f^{(\tau)}(w^k) )\|^2 \\ 
		&\leq& (1-\delta)(1+\beta) \|e^k_\tau\|^2 + (1-\delta)\left(  1 + \frac{1}{\beta}   \right) \frac{\eta^2}{{\cal L}_1^2} \| \nabla f^{(\tau)}(x^k) - \nabla f^{(\tau)}(w^k) \|^2 \\ 
		&& \quad + \frac{\eta^2}{{\cal L}_1^2}(1-\delta) \mathbb{E}_k\|g^k_\tau - (\nabla f^{(\tau)}(x^k) - \nabla f^{(\tau)}(w^k) )\|^2 \\ 
		&\leq& \left(  1 - \frac{\delta}{2}  \right) \|e^k_\tau\|^2 + \frac{2(1-\delta)\eta^2}{\delta {\cal L}_1^2}   \| \nabla f^{(\tau)}(x^k) - \nabla f^{(\tau)}(w^k) \|^2 \\
		&&  \quad + \frac{\eta^2}{{\cal L}_1^2}(1-\delta) \mathbb{E}_k\|g^k_\tau - (\nabla f^{(\tau)}(x^k) - \nabla f^{(\tau)}(w^k) )\|^2, 
	\end{eqnarray*}
	where we choose $\beta = \frac{\delta}{2(1-\delta)}$ when $\delta<1$. When $\delta=1$, it is easy to see the above inequality also holds. 	Since 
	\begin{eqnarray*}
		 \mathbb{E}_k\|g^k_\tau - (\nabla f^{(\tau)}(x^k) - \nabla f^{(\tau)}(w^k) )\|^2 
		&\leq& \mathbb{E}_k \|g^k_\tau\|^2 \\ 
		&\leq& 2L \mathbb{E}_k (f^{(\tau)}_{i^\tau_k}(w^k) - f^{(\tau)}_{i^\tau_k}(x^k) - \langle \nabla f^{(\tau)}_{i^\tau_k}(x^k), w^k-x^k \rangle )\\ 
		&=& 2L (f^{(\tau)}(w^k) - f^{(\tau)}(x^k) - \langle \nabla f^{(\tau)}(x^k), w^k-x^k \rangle ), 
	\end{eqnarray*}
	
	and 
	$$
	\| \nabla f^{(\tau)}(x^k) - \nabla f^{(\tau)}(w^k) \|^2 \leq 2{\bar L} (f^{(\tau)}(w^k) - f^{(\tau)}(x^k) - \langle \nabla f^{(\tau)}(x^k), w^k-x^k \rangle ), 
	$$
	we arrive at 
	\begin{eqnarray*}
 \mathbb{E}_k[\|e^{k+1}_\tau \|^2] 
	&\leq&  \left(  1 - \frac{\delta}{2}  \right) \|e^k_\tau\|^2 + \frac{2(1-\delta)\eta^2}{{\cal L}_1^2} \left(  \frac{2{\bar L}}{\delta} + L  \right) (f^{(\tau)}(w^k) - f^{(\tau)}(x^k) - \langle \nabla f^{(\tau)}(x^k), w^k-x^k \rangle ). 
	\end{eqnarray*}

	Therefore, 
	\begin{eqnarray*}
		 \mathbb{E}_k \left[  \frac{1}{n}\sum_{\tau=1}^n \|e^{k+1}_\tau \|^2   \right] 		&\leq&  \left(  1 - \frac{\delta}{2}  \right) \frac{1}{n} \sum_{\tau=1}^n \|e^k_\tau\|^2 \\ 
		&& + \frac{2(1-\delta)\eta^2}{{\cal L}_1^2} \left(  \frac{2{\bar L}}{\delta} + L  \right) \frac{1}{n} \sum_{\tau=1}^n (f^{(\tau)}(w^k) - f^{(\tau)}(x^k) - \langle \nabla f^{(\tau)}(x^k), w^k-x^k \rangle ) \\
		&=&  \left(  1 - \frac{\delta}{2}  \right) \frac{1}{n} \sum_{\tau=1}^n \|e^k_\tau\|^2 + \frac{2(1-\delta)\eta^2}{{\cal L}_1^2} \left(  \frac{2{\bar L}}{\delta} + L  \right)  (f(w^k) - f(x^k) - \langle \nabla f(x^k), w^k-x^k \rangle ) .
	\end{eqnarray*}

	\subsection{Proof of Lemma \ref{lm:ek+1-2}}

	Under Assumption \ref{as:expcompressor}, we have $\mathbb{E}[Q(x)] = \delta x$, and 
	\begin{eqnarray*}
		 \mathbb{E}_k\|e^{k+1}\|^2 
		&=& \mathbb{E}_k \left\| \frac{1}{n} \sum_{\tau=1}^n e^{k+1}_{\tau} \right\|^2 \\
		&=& \frac{1}{n^2} \sum_{i, j} \mathbb{E}_k \langle e^{k+1}_i, e^{k+1}_j \rangle \\ 
		&=& \frac{1}{n^2} \sum_{\tau=1}^n \mathbb{E}_k \|e^{k+1}_{\tau}\|^2 + \frac{1}{n^2} \sum_{i\neq j} \mathbb{E}_k \langle e^{k+1}_i, e^{k+1}_j \rangle \\
		&\leq& \frac{1-\delta}{n^2} \sum_{\tau=1}^n \mathbb{E}_k \left\|e^k_{\tau} + \frac{\eta}{{\cal L}_1} g^k_{\tau} \right\|^2 + \frac{(1-\delta)^2}{n^2} \sum_{i\neq j} \mathbb{E}_k \left\langle e^k_i + \frac{\eta}{{\cal L}_1}  g^k_i , e^k_j + \frac{\eta}{{\cal L}_1} g^k_j \right\rangle \\ 
		&=& \frac{(1-\delta)^2}{n^2} \mathbb{E}_k \left\|\sum_{\tau=1}^n (e^k_{\tau} + \frac{\eta}{{\cal L}_1} g^k_{\tau}) \right\|^2 + \frac{(1-\delta)\delta}{n^2} \sum_{\tau=1}^n \mathbb{E}_k \left\|e^k_{\tau} + \frac{\eta}{{\cal L}_1} g^k_{\tau} \right\|^2 \\ 
		&\leq& (1-\delta) \mathbb{E}_k \left\|e^k + \frac{\eta}{{\cal L}_1} g^k \right\|^2 + \frac{(1-\delta)\delta}{n^2} \sum_{\tau=1}^n \mathbb{E}_k \left\|e^k_{\tau} + \frac{\eta}{{\cal L}_1} g^k_{\tau} \right\|^2. 
	\end{eqnarray*}
	
	Under Assumption \ref{as:topkcompressor}, we have 
	\begin{eqnarray*}
		\mathbb{E}_k \|e^{k+1}\|^2 &=&  \mathbb{E}_k \left\| \frac{1}{n} \sum_{\tau=1}^n e^{k+1}_{\tau} \right\|^2 \\
		&=& \mathbb{E}_k  \left\| \frac{1}{n} \sum_{\tau=1}^n \left( e^k_{\tau} + \frac{\eta}{{\cal L}_1} g^k_{\tau} - Q\left(\frac{\eta}{{\cal L}_1} g^k_{\tau} + e^k_{\tau} \right) \right) \right\|^2 \\
		&\overset{Assumption \ref{as:topkcompressor}}{\leq}&(1 - \delta^\prime) \mathbb{E}_k \left\| e^k + \frac{\eta}{{\cal L}_1} g^k \right\|^2 \\ 
		&\leq&  (1-\delta) \mathbb{E}_k \left\|e^k + \frac{\eta}{{\cal L}_1} g^k \right\|^2. 
	\end{eqnarray*}
	
	Overall, under Assumption \ref{as:expcompressor} or Assumption \ref{as:topkcompressor}, we have 
	\begin{eqnarray*}
		\mathbb{E}_k \|e^{k+1}\|^2 &\leq&  (1-\delta) \mathbb{E}_k \left\|e^k + \frac{\eta}{{\cal L}_1} g^k \right\|^2 + \frac{(1-\delta)\delta}{n^2} \sum_{\tau=1}^n \mathbb{E}_k \left\|e^k_{\tau} + \frac{\eta}{{\cal L}_1} g^k_{\tau} \right\|^2 \\ 
		&\leq& (1-\delta) \mathbb{E}_k \left\|e^k + \frac{\eta}{{\cal L}_1} g^k \right\|^2 + \frac{2(1-\delta)\delta}{n^2} \sum_{\tau=1}^n \|e^k_{\tau} \|^2  + \frac{2(1-\delta)\delta \eta^2 }{n^2 {\cal L}_1^2} \sum_{\tau=1}^n \mathbb{E}_k \| g^k_{\tau} \|^2 .
	\end{eqnarray*}

The first term on the right hand side above can be bounded as
	\begin{eqnarray*}
 (1-\delta)\mathbb{E}_k \left\|e^k + \frac{\eta}{{\cal L}_1} g^k \right\|^2 	&=& (1-\delta) \mathbb{E}_k \left\|e^k + \frac{\eta}{{\cal L}_1} (\nabla f(x^k) - \nabla f(w^k)) + \frac{\eta}{{\cal L}_1} g^k - \frac{\eta}{{\cal L}_1} (\nabla f(x^k) - \nabla f(w^k)) \right\|^2 \\ 
		&=& (1-\delta) \mathbb{E}_k \left\|e^k + \frac{\eta}{{\cal L}_1} (\nabla f(x^k) - \nabla f(w^k)) \right\|^2  + (1-\delta)\frac{\eta^2}{{\cal L}_1^2} \mathbb{E}_k \|g^k - (\nabla f(x^k) - \nabla f(w^k))\|^2 \\ 
		&\leq& \left(  1 - \frac{\delta}{2}  \right) \|e^k\|^2 + \frac{2(1-\delta)\eta^2}{\delta {\cal L}_1^2} \|\nabla f(x^k) - \nabla f(w^k)\|^2 \\ 
		&& + (1-\delta)\frac{\eta^2}{{\cal L}_1^2} \mathbb{E}_k \|g^k - (\nabla f(x^k) - \nabla f(w^k))\|^2 \\ 
		&\overset{(\ref{eq:expL2})}{\leq}&  \left(  1 - \frac{\delta}{2}  \right) \|e^k\|^2 + \frac{2(1-\delta)\eta^2}{\delta {\cal L}_1^2} \|\nabla f(x^k) - \nabla f(w^k)\|^2 \\ 
		&& +  (1-\delta)\frac{2L\eta^2}{n{\cal L}_1^2}  (f(w^k) - f(x^k) - \langle \nabla f(x^k), w^k-x^k \rangle ) \\ 
		&\leq& \left(  1 - \frac{\delta}{2}  \right) \|e^k\|^2  + \frac{2(1-\delta)\eta^2}{{\cal L}_1^2} \left(  \frac{2L_f}{\delta} + \frac{L}{n}  \right) (f(w^k) - f(x^k) - \langle \nabla f(x^k), w^k-x^k \rangle ). 
	\end{eqnarray*}
	
	Moreover, 
	\begin{eqnarray*}
		\frac{1}{n} \sum_{\tau=1}^n \mathbb{E}_k \|g^k_\tau\|^2 &\leq& \frac{2L}{n} \sum_{\tau=1}^n \mathbb{E}_k ( f^{(\tau)}_{i_k^\tau}(w^k) - f^{(\tau)}_{i_k^\tau}(x^k) - \langle \nabla f^{(\tau)}_{i_k^\tau}(x^k), w^k-x^k \rangle ) \\ 
		&=& 2L  (f(w^k) - f(x^k) - \langle \nabla f(x^k), w^k-x^k \rangle ). 
	\end{eqnarray*}
	
	Hence, 
	\begin{eqnarray*}
		 \mathbb{E}_k \|e^{k+1}\|^2 
		&\leq&  \left(  1 - \frac{\delta}{2}  \right) \|e^k\|^2 + \frac{2(1-\delta)\delta}{n^2} \sum_{\tau=1}^n \|e^k_{\tau} \|^2 \\ 
		&& +  \frac{2(1-\delta)\eta^2}{{\cal L}_1^2} \left(  \frac{2L_f}{\delta} + \frac{L}{n} + \frac{2L\delta}{n} \right) (f(w^k) - f(x^k) - \langle \nabla f(x^k), w^k-x^k \rangle ) \\
		&\leq&  \left(  1 - \frac{\delta}{2}  \right) \|e^k\|^2 + \frac{2(1-\delta)\delta}{n^2} \sum_{\tau=1}^n \|e^k_{\tau} \|^2 \\ 
		&& +  \frac{2(1-\delta)\eta^2}{{\cal L}_1^2} \left(  \frac{2L_f}{\delta} + \frac{3L}{n} \right) (f(w^k) - f(x^k) - \langle \nabla f(x^k), w^k-x^k \rangle ).
	\end{eqnarray*}

	\section{Proof of Theorem \ref{th:eclkatyusha-1}}

	From $\|e^k\|^2 \leq \frac{1}{n} \sum_{\tau=1}^n \|e^k_\tau\|^2$, Equation \eqref{eq:wk+1}, and Lemma \ref{lm:zyk+1}, we can obtain 
	\begin{eqnarray*}
	 \mathbb{E}_k\left[  {\tilde {\cal Z}}^{k+1} + {\cal Y}^{k+1} + {\cal W}^{k+1} \right] 
		&\leq& \frac{{\cal L}_1{\tilde {\cal Z}}^k}{{\cal L}_1 + \eta\mu/2} + (1-\theta_1 - \theta_2 + \frac{\theta_2}{q}) {\cal Y}^k + (1-p+pq){\cal W}^k \\ 
		&& \quad + \left(   \frac{{\cal L}_1}{2\eta} + \frac{\mu_f}{2}  \right) \frac{1}{n} \sum_{\tau=1}^n \|e^k_\tau\|^2 + \left(   \frac{{\cal L}_1}{2\eta} + \frac{\mu}{2}  \right) \mathbb{E}_k \frac{1}{n} \sum_{\tau=1}^n \|e^{k+1}_\tau\|^2 \\ 
		&&  \quad  -  \frac{1}{\theta_1} \left(  \theta_2 - \frac{2L}{n{\cal L}_1}  \right) (f(w^k) - f(x^k) - \langle \nabla f(x^k), w^k-x^k \rangle ) \\ 
		&\overset{\text{Lemma}~\ref{lm:ek+1}}{\leq}& \frac{{\cal L}_1{\tilde {\cal Z}}^k}{{\cal L}_1 + \eta\mu/2} + (1-\theta_1 - \theta_2 + \frac{\theta_2}{q}) {\cal Y}^k + (1-p+pq){\cal W}^k \\ 
		&& \quad + \left(   \frac{{\cal L}_1}{\eta} + \mu \right) \frac{1}{n} \sum_{\tau=1}^n \|e^k_\tau\|^2 \\ 
		&& \quad -  \left( \frac{1}{\theta_1} \left(  \theta_2 - \frac{2L}{n{\cal L}_1}  \right) - \frac{2(1-\delta)\eta^2}{{\cal L}_1^2} \left(  \frac{2{\bar L}}{\delta} + L  \right) \left(  \frac{{\cal L}_1}{2\eta} + \frac{\mu}{2}  \right) \right) \\ 
		&& \quad \cdot (f(w^k) - f(x^k) - \langle \nabla f(x^k), w^k-x^k \rangle ) \\ 
		&\leq& \frac{{\cal L}_1{\tilde {\cal Z}}^k}{{\cal L}_1 + \eta\mu/2} + (1-\theta_1 - \theta_2 + \frac{\theta_2}{q}) {\cal Y}^k + (1-p+pq){\cal W}^k \\ 
		&& \quad +  \frac{4{\cal L}_1}{3\eta} \frac{1}{n} \sum_{\tau=1}^n \|e^k_\tau\|^2 - \frac{1}{\theta_1} \left(  \theta_2 - \frac{1}{{\cal L}_1} \left(  \frac{2L}{n} + \frac{8(1-\delta) {\bar L}}{9\delta} + \frac{4(1-\delta) L}{9}  \right)  \right) \\ 
		&& \quad \cdot  (f(w^k) - f(x^k) - \langle \nabla f(x^k), w^k-x^k \rangle ). 
	\end{eqnarray*}
	
	Therefore, from Lemma \ref{lm:ek+1}, we have 
	\begin{eqnarray*}
		&& \mathbb{E}_k\left[  {\tilde {\cal Z}}^{k+1} + {\cal Y}^{k+1} + {\cal W}^{k+1} + \frac{4{\cal L}_1}{\delta \eta} \cdot \frac{1}{n} \sum_{\tau=1}^n \|e^{k+1}\|^2 \right] \\ 
		&\leq& \frac{{\cal L}_1{\tilde {\cal Z}}^k}{{\cal L}_1 + \eta\mu/2} + (1-\theta_1 - \theta_2 + \frac{\theta_2}{q}) {\cal Y}^k + (1-p+pq){\cal W}^k \\ 
		&& + \left(  1 - \frac{\delta}{6}  \right)\frac{4{\cal L}_1}{\delta \eta} \cdot \frac{1}{n} \sum_{\tau=1}^n \|e^k\|^2  - \frac{1}{\theta_1} \left(  \theta_2 - \frac{1}{{\cal L}_1} \left(   \frac{2L}{n} +  \frac{56(1-\delta) {\bar L}}{9\delta^2} + \frac{28(1-\delta) L}{9\delta}   \right)  \right) \\ 
		&& \cdot (f(w^k) - f(x^k) - \langle \nabla f(x^k), w^k-x^k \rangle ) \\ 
		&=& \frac{{\cal L}_1{\tilde {\cal Z}}^k}{{\cal L}_1 + \eta\mu/2} + (1-\theta_1 - \theta_2 + \frac{\theta_2}{q}) {\cal Y}^k + (1-p+pq){\cal W}^k + \left(  1 - \frac{\delta}{6}  \right)\frac{4{\cal L}_1}{\delta \eta} \cdot \frac{1}{n} \sum_{\tau=1}^n \|e^k\|^2 \\ 
		&& - \frac{1}{\theta_1} \left(  \theta_2 - \frac{{\cal L}_2}{2{\cal L}_1}  \right) (f(w^k) - f(x^k) - \langle \nabla f(x^k), w^k-x^k \rangle ). 
	\end{eqnarray*}
	
	When $\theta_2 \geq \frac{{\cal L}_2}{2{\cal L}_1}$ we can get the result.

	\section{Proof of Theorem \ref{th:eclkatyusha-2}} 
	
	From Lemma \ref{lm:zyk+1} and (\ref{eq:wk+1}), we have

	\begin{eqnarray*}
		&& \mathbb{E}_k\left[  {\tilde {\cal Z}}^{k+1} + {\cal Y}^{k+1} + {\cal W}^{k+1} \right] \\ 
		&\leq & \frac{{\cal L}_1{\tilde {\cal Z}}^k}{{\cal L}_1 + \eta\mu/2} + \left(1-\theta_1 - \theta_2 + \frac{\theta_2}{q} \right) {\cal Y}^k + (1-p+pq){\cal W}^k \\ 
		&& \quad + \left(   \frac{{\cal L}_1}{2\eta} + \frac{\mu_f}{2}  \right) \|e^k\|^2 + \left(   \frac{{\cal L}_1}{2\eta} + \frac{\mu}{2}  \right) \mathbb{E}_k \|e^{k+1}\|^2  \\ 
		&& \quad  -  \frac{1}{\theta_1} \left(  \theta_2 - \frac{2L}{n{\cal L}_1}  \right) (f(w^k) - f(x^k) - \langle \nabla f(x^k), w^k-x^k \rangle ) \\ 
		&\overset{\text{Lemma}~\ref{lm:ek+1-2}}{\leq}& \frac{{\cal L}_1{\tilde {\cal Z}}^k}{{\cal L}_1 + \eta\mu/2} + \left(1-\theta_1 - \theta_2 + \frac{\theta_2}{q} \right) {\cal Y}^k + (1-p+pq){\cal W}^k \\ 
		&& \quad + \left(   \frac{{\cal L}_1}{\eta} + \mu  \right) \|e^k\|^2 + \left(   \frac{{\cal L}_1}{2\eta} + \frac{\mu}{2}  \right) \frac{2(1-\delta) \delta}{n} \cdot \frac{1}{n} \sum_{\tau=1}^n \|e^k_\tau\|^2 \\ 
		&& \quad  -  \left( \frac{1}{\theta_1} \left(  \theta_2 - \frac{2L}{n{\cal L}_1}  \right) - \frac{4(1-\delta) \eta}{3{\cal L}_1} \left(  \frac{2L_f}{\delta} + \frac{3L}{n}  \right) \right)   (f(w^k) - f(x^k) - \langle \nabla f(x^k), w^k-x^k \rangle ) \\ 
		&\leq & \frac{6\theta_1{\cal L}_1{\tilde {\cal Z}}^k}{ 6\theta_1{\cal L}_1 + \mu} + \left(1-\theta_1 - \theta_2 + \frac{\theta_2}{q} \right) {\cal Y}^k + (1-p+pq){\cal W}^k + \frac{4{\cal L}_1}{3\eta} \|e^k\|^2 \\ 
		&&  \quad + \frac{4{\cal L}_1(1-\delta) \delta}{3\eta n} \cdot \frac{1}{n} \sum_{\tau=1}^n \|e^k_\tau\|^2  \\
		&& \quad - \frac{1}{\theta_1} \left(   \theta_2 - \frac{2L}{n{\cal L}_1} - \frac{4(1-\delta)}{9{\cal L}_1} \left(  \frac{2L_f}{\delta} + \frac{3L}{n}  \right)  \right)   (f(w^k) - f(x^k) - \langle \nabla f(x^k), w^k-x^k \rangle ). 
	\end{eqnarray*}

	Therefore, from Lemma \ref{lm:ek+1} and Lemma \ref{lm:ek+1-2}, we can get 
	
	\begin{eqnarray*}
		&& \mathbb{E}_k\left[  {\tilde {\cal Z}}^{k+1} + {\cal Y}^{k+1} + {\cal W}^{k+1} + \frac{4{\cal L}_1}{\delta \eta} \|e^{k+1}\|^2  + \frac{28{\cal L}_1(1-\delta)}{\delta \eta n} \cdot \frac{1}{n} \sum_{\tau=1}^n \|e^{k+1}_\tau\|^2  \right] \\ 
		&\leq&  \frac{6\theta_1{\cal L}_1{\tilde {\cal Z}}^k}{ 6\theta_1{\cal L}_1 + \mu} + (1-\theta_1 - \theta_2 + \frac{\theta_2}{q}) {\cal Y}^k + (1-p+pq){\cal W}^k \\ 
		&& + \left(  1 - \frac{\delta}{6}  \right) \frac{4{\cal L}_1}{\delta \eta} \|e^k\|^2 + \left(  1 - \frac{\delta}{6}  \right) \frac{28{\cal L}_1(1-\delta)}{\delta \eta n} \cdot \frac{1}{n} \sum_{\tau=1}^n \|e^{k}_\tau\|^2 \\
		&& - \frac{1}{\theta_1} \left(   \theta_2 - \frac{2L}{n{\cal L}_1} - \frac{28(1-\delta)}{9\delta {\cal L}_1} \left(  \frac{2L_f}{\delta} + \frac{3L}{n}  \right)  - \frac{56(1-\delta)}{3\delta {\cal L}_1} \left(  \frac{2{\bar L}}{\delta n} + \frac{L}{n}  \right) \right) \\ 
		&& \cdot  (f(w^k) - f(x^k) - \langle \nabla f(x^k), w^k-x^k \rangle ) \\ 
		&\leq&  \frac{6\theta_1{\cal L}_1{\tilde {\cal Z}}^k}{ 6\theta_1{\cal L}_1 + \mu} + (1-\theta_1 - \theta_2 + \frac{\theta_2}{q}) {\cal Y}^k + (1-p+pq){\cal W}^k \\ 
		&& + \left(  1 - \frac{\delta}{6}  \right) \frac{4{\cal L}_1}{\delta \eta} \|e^k\|^2 + \left(  1 - \frac{\delta}{6}  \right) \frac{28{\cal L}_1(1-\delta)}{\delta \eta n} \cdot \frac{1}{n} \sum_{\tau=1}^n \|e^{k}_\tau\|^2 \\
		&& - \frac{1}{\theta_1} \left(   \theta_2 - \frac{2L}{n{\cal L}_1} - \frac{392(1-\delta) L_f}{9\delta^2 {\cal L}_1} - \frac{28(1-\delta) L}{\delta {\cal L}_1 n} \right) \\ 
		&& \cdot (f(w^k) - f(x^k) - \langle \nabla f(x^k), w^k-x^k \rangle ) \\ 
		&=&  \frac{6\theta_1{\cal L}_1{\tilde {\cal Z}}^k}{ 6\theta_1{\cal L}_1 + \mu} + (1-\theta_1 - \theta_2 + \frac{\theta_2}{q}) {\cal Y}^k + (1-p+pq){\cal W}^k \\ 
		&& + \left(  1 - \frac{\delta}{6}  \right) \frac{4{\cal L}_1}{\delta \eta} \|e^k\|^2 + \left(  1 - \frac{\delta}{6}  \right) \frac{28{\cal L}_1(1-\delta)}{\delta \eta n} \cdot \frac{1}{n} \sum_{\tau=1}^n \|e^{k}_\tau\|^2 \\
		&& - \frac{1}{\theta_1} \left(   \theta_2 - \frac{{\cal L}_3}{2{\cal L}_1} \right)  (f(w^k) - f(x^k) - \langle \nabla f(x^k), w^k-x^k \rangle ) ,
	\end{eqnarray*}
	where we use ${\bar L} \leq n L_f$ in the second inequality. When $\theta_2 \geq  \frac{{\cal L}_3}{2{\cal L}_1}$ we can get the result.

	\section{Proof of Corollary \ref{co:eclkatyusha}}

\begin{itemize}

\item[(i)] First, we have $\frac{1}{2} \geq \theta_2 \geq \frac{{\cal L}_2}{2{\cal L}_1}$. Form the definition of $\theta_1$, we know $\theta_1 \leq \frac{1}{2}$. Hence $\theta_1 + \theta_2 \leq 1$. Next we discuss two cases:
\begin{itemize} 
\item	 {\bf Case 1: $3\mu \eta < {\cal L}_1$.} In this case, we have ${\cal L}_1 = \max\{  {\cal L}_4, L_f  \}$. Then from Theorem \ref{th:eclkatyusha-1} and same as the proof of Theorem 3.2 in \cite{qian2019svrg}, we have $\mathbb{E}[\Phi^k] \leq \epsilon \Phi^0$ as long as 
	$$
	k \geq O\left(  \frac{1}{\delta} + \frac{1}{p} + \sqrt{\frac{L_f}{\mu}} + \sqrt{\frac{{\cal L}_4}{\mu p}}  \right). 
	$$
	
	Since ${\cal L}_4 = {\cal L}_2$, we can get the result. 
	
\item 	{\bf Case 2: $3\mu \eta = {\cal L}_1$.} In this case, we have 
	$$
	\frac{\mu}{\mu + 6\theta_1 {\cal L}_1} = \frac{\mu}{\mu + 6\mu} = \frac{1}{7} \geq \frac{p}{7}. 
	$$
	
	Hence, from Theorem \ref{th:eclkatyusha-1} and same as the proof of Theorem 3.2 in \cite{qian2019svrg}, we also have $\mathbb{E}[\Phi^k] \leq \epsilon \Phi^0$ for  
	$$
	k \geq O\left(  \frac{1}{\delta} + \frac{1}{p} + \sqrt{\frac{L_f}{\mu}} + \sqrt{\frac{{\cal L}_4}{\mu p}}  \right). 
	$$
	Since ${\cal L}_4 = {\cal L}_2$, we can get the result. 

\end{itemize}
	
\item[(ii)] By using Theorem \ref{th:eclkatyusha-2}, same as (i), we can get the result.

\end{itemize}

\end{document}